\documentclass[leqno,11pt]{amsart}

\usepackage{amssymb,
			enumitem, 
			caption, 
			array, 
			tikz, 
            xy 
}

\xyoption{all}

\DeclareCaptionFormat{myfigformat}{#1#2#3\hrulefill}
\DeclareCaptionFormat{mytabformat}{~ \\[10pt] #1#2#3\hrulefill}
\usepackage[
	bookmarksnumbered,
	pdftitle={quantum dilog AA}
	pdfauthor={Justin Allman}
	]
	{hyperref}
	
\usepackage{geometry}
\makeatletter
\newcommand{\leqnomode}{\tagsleft@true}
\newcommand{\reqnomode}{\tagsleft@false}
\makeatother

\newcommand{\curly}{\mathcal}
\newcommand{\G}{\boldsymbol{\mathrm{G}}}
\newcommand{\M}{\boldsymbol{\mathrm{Rep}}}
\newcommand{\Mat}{\mathrm{Mat}}
\DeclareMathOperator{\Hom}{Hom}
\DeclareMathOperator{\Ext}{Ext}
\DeclareMathOperator{\dhom}{hom}
\DeclareMathOperator{\dext}{ext}
\DeclareMathOperator{\codim}{codim}
\newcommand{\Hor}{\mathrm{Hor}}
\newcommand{\Ver}{\mathrm{Ver}}
\newcommand{\GL}{\mathrm{GL}}
\newcommand{\U}{\mathrm{U}}
\newcommand{\kp}{\vdash}
\newcommand{\ddim}{\underline{\dim}}
\newcommand{\cqmod}{\C Q\mathtt{-mod}}

\newcommand{\w}{\mathrm{w}}
\newcommand{\dd}{\mathbf{down}}
\newcommand{\rr}{\mathbf{right}}
\newcommand{\supc}{\mathrm{sc}}
\newcommand{\coho}{\mathrm{H}}
\newcommand{\vip}{\mathbf{vip}}
\newcommand{\hip}{\mathbf{hip}}
\newcommand{\yy}{\mathbf{y}}
\newcommand{\cP}{\curly{P}}

\newcommand{\hpr}{\mathtt{H}\Phi}
\newcommand{\vpr}{\mathtt{V}\Phi}

\newcommand{\hpri}[1]{\Phi(#1,\bullet)}
\newcommand{\vprj}[1]{\Phi(\bullet,#1)}

\newcommand{\hts}{\mathtt{HorT}}
\newcommand{\hhs}{\mathtt{HorH}}
\newcommand{\vts}{\mathtt{VerT}}
\newcommand{\vhs}{\mathtt{VerH}}
\newcommand{\rowQ}[1]{Q(#1,\bullet)}
\newcommand{\colQ}[1]{Q(\bullet,#1)}

\newcommand{\Neg}{\mathtt{Neg}}
\newcommand{\Pos}{\mathtt{Pos}}

\newcommand{\downform}[2]
	{\left\langle \begin{smallmatrix} #1 \\ \downarrow \\ #2 \end{smallmatrix} \right\rangle}
\newcommand{\upform}[2]{\left\langle \begin{smallmatrix} #1 \\ \uparrow \\ #2 \end{smallmatrix} \right\rangle}
\newcommand{\rightform}[2]{\langle #1 \rightarrow #2 \rangle}
\newcommand{\leftform}[2]{\langle #1 \leftarrow #2 \rangle}

\newcommand{\smallhstrata}[6]{
	\begin{smallmatrix}
		& #1	&	\\
	#2	&		&#3	\\
	#4	&		&#5	\\
		& #6	&	
	\end{smallmatrix}
	}
\newcommand{\smallvstrata}[6]{
	\begin{smallmatrix}
		& #2	& #4	&	\\
	#1	&		&		&#6	\\
		& #3	& #5	&
	\end{smallmatrix}
	}
\newcommand{\hstrata}[6]{
	\begin{array}{lcr}
		& #1	&	\\
	#2	&		&#3	\\
	#4	&		&#5	\\
		& #6	&	
	\end{array}
	}

\newcommand{\bfs}[1]{\mathbf{#1}}

\newcommand{\R}{\mathbb{R}}
\newcommand{\Q}{\mathbb{Q}}
\newcommand{\Z}{\mathbb{Z}}
\newcommand{\C}{\mathbb{C}}
\newcommand{\E}{\mathbb{E}}
\newcommand{\N}{\mathbb{N}}
\newcommand{\A}{\mathbb{A}}

\newcommand{\union}{\cup}
\newcommand{\Union}{\bigcup}
\newcommand{\intersect}{\cap}

\newcommand{\dirsum}{\oplus}
\newcommand{\Dirsum}{\bigoplus}

\newcommand{\Tensor}{\bigotimes}
\newcommand{\compose}{\circ}
\newcommand{\til}{\widetilde}
\newcommand{\tr}{\mathrm{Tr}}
\newcommand{\iso}{\cong}
\newcommand{\homeo}{\approx}
\newcommand{\hmtpc}{\simeq}
\newcommand{\floor}[1]{\left\lfloor #1 \right\rfloor}

\newcommand{\defin}[1]{\textbf{\texttt{#1}}}

\theoremstyle{plain}
\newtheorem{prop}{Proposition}[section]
\newtheorem{lem}[prop]{Lemma}

\newtheorem{cor}[prop]{Corollary}
\newtheorem{thm}[prop]{Theorem}
\newtheorem*{thm*}{Theorem}

\theoremstyle{definition}
\newtheorem{defn}[prop]{Definition}

\theoremstyle{remark}
\newtheorem{remark}[prop]{Remark}
\newtheorem{example}[prop]{Example}

\title[{q}-dilog identities for square products of {A}-type]{Quantum dilogarithm identities for the square product of {A}-type {D}ynkin quivers}

\author[Allman]{Justin Allman}
\address{Department of Mathematics \\ US Naval Academy \\ Annapolis, MD}
\email{allman@usna.edu}

\author[Rim{\'a}nyi]{Rich\'ard Rim\'anyi}
\address{Department of Mathematics, UNC--Chapel Hill \\ Chapel Hill, NC}
\email{rimanyi@email.unc.edu}

\subjclass[2010]{16G20,05E10}
\keywords{Quantum dilogarithm, Donaldson--Thomas invariant, quiver with potential, rapid decay cohomology}

\begin{document}

\maketitle


\begin{abstract}
The famous pentagon identity for quantum dilogarithms has a generalization for every Dynkin quiver, due to Reineke. A more advanced generalization is associated with a pair of alternating Dynkin quivers, due to Keller.
The description and proof of Keller's identities involves cluster algebras and cluster categories, and the statement of the identity is implicit. In this paper we describe Keller's identities explicitly, and prove them by a dimension counting argument. Namely, we consider quiver representations $\M_\gamma$ together with a superpotential function $W_\gamma$, and calculate the Betti numbers of the equivariant $W_\gamma$ rapid decay cohomology algebra of $\M_\gamma$ in two different ways corresponding to two natural stratifications of $\M_\gamma$. This approach is suggested by Kontsevich and Soibelman in relation with the Cohomological Hall Algebra of quivers, and the associated Donaldson--Thomas invariants.
\end{abstract}

\tableofcontents


\section{Introduction}
\label{s:intro}

Define $\cP_n=1/\prod_{i=1}^n(1-q^i)$. Write $(m_{10},m_{01},m_{11})\kp (\gamma_1,\gamma_2)$ if $m_{10}+m_{11}=\gamma_1$ and $m_{01}+m_{11}=\gamma_2$ (all non-negative integers). The remarkable identity
\begin{equation}\label{eqn:5term}
\cP_{\gamma_1}\cP_{\gamma_2}=
\sum_{(m_{10},m_{01},m_{11})\kp (\gamma_1,\gamma_2)}
q^{m_{10}m_{01}}
\cP_{m_{10}}\cP_{m_{01}}\cP_{m_{11}},
\end{equation}
named {\em pentagon identity}, has several interpretations in mathematics.

In combinatorics it is equivalent to the Durfee's square identity which is an effective way of counting partitions going back to at least Cauchy, see \cite{rraway2018} and references therein. In analysis (and number theory) it is a quantum version of the five-term identity for the dilogarithm function, see \cite{lfrk1994,dz1988,dz2007} and references therein. In geometry the common value of the two sides is called the Donaldson--Thomas (DT) invariant associated with the $A_2$ quiver, or it is interpreted as the simplest wall-crossing formula for DT-invariants, see \cite{mkys2014}. In topology, Equation \eqref{eqn:5term} is interpreted as two ways of counting the Betti numbers of the $GL_{\gamma_1}(\C)\times GL_{\gamma_2}(\C)$-equivariant cohomology of $\Hom(\C^{\gamma_1},\C^{\gamma_2})$---on the left hand side one uses the fact that $\Hom(\C^{\gamma_1},\C^{\gamma_2})$ is contractible, and on the right hand side one cuts the space $\Hom(\C^{\gamma_1},\C^{\gamma_2})$ into orbits, see \cite{mkys2011,mkys2014,rr2013}.

It is remarkable that the identity \eqref{eqn:5term} {\em for all $\gamma_1,\gamma_2$ together} can be encoded as a single identity
\begin{equation}\label{eqn:E5}
\E(y_{1})\E(y_{2})=\E(y_{2})\E(y_{12})\E(y_{1}),
\end{equation}
where $\E$ is an explicit power series called quantum dilogarithm series, and the $y$'s are certain non-commuting variables, see Section \ref{s:q.dilog.series}.

There are natural generalizations of the identities \eqref{eqn:5term}, \eqref{eqn:E5} to all Dynkin quivers, due to Reineke \cite{mr2010}, and their various interpretations mentioned above are well studied. The topic of this paper is a higher level of quantum dilogarithm identities which Keller found in relation to cluster categories \cite{bk2011,bk2013.fpsac}. These identities are parameterized by pairs of Dynkin diagrams of type A, D, E. Next we explain the simplest case of Keller's identities, the so-called $A_2\square A_2$ case.

From \eqref{eqn:5term} by formal manipulation (squaring) one obtains
\begin{multline}\label{eqn:55term}
\mathop{\sum_{(m_{10},m_{01},m_{11})\kp (\gamma_1,\gamma_2)}}_{(n_{10},n_{01},n_{11})\kp (\gamma_3,\gamma_4)}
q^{m_{10}m_{01}+n_{10}n_{01}}
\cP_{m_{10}}\cP_{m_{01}}\cP_{m_{11}}\cP_{n_{10}}\cP_{n_{01}}\cP_{n_{11}}=\\
\mathop{\sum_{(m_{10},m_{01},m_{11})\kp (\gamma_1,\gamma_3)}}_{(n_{10},n_{01},n_{11})\kp (\gamma_2,\gamma_4)}
q^{m_{10}m_{01}+n_{10}n_{01}}
\cP_{m_{10}}\cP_{m_{01}}\cP_{m_{11}}\cP_{n_{10}}\cP_{n_{01}}\cP_{n_{11}}.
\end{multline}
In fact both sides are equal to $\cP_{\gamma_1}\cP_{\gamma_2}\cP_{\gamma_3}\cP_{\gamma_4}$. The novelty of Keller's identity for $A_2 \square A_2$ is that in \eqref{eqn:55term} one can insert an extra factor in each term, namely
\begin{multline}\label{eqn:55termKeller}
\mathop{\sum_{(m_{10},m_{01},m_{11})\kp (\gamma_1,\gamma_2)}}_{(n_{10},n_{01},n_{11})\kp (\gamma_3,\gamma_4)}
q^{m_{10}m_{01}+n_{10}n_{01}+m_{11}n_{11}}
\cP_{m_{10}}\cP_{m_{01}}\cP_{m_{11}}\cP_{n_{10}}\cP_{n_{01}}\cP_{n_{11}}= \\
\mathop{\sum_{(m_{10},m_{01},m_{11})\kp (\gamma_1,\gamma_3)}}_{(n_{10},n_{01},n_{11})\kp (\gamma_2,\gamma_4)}
q^{m_{10}m_{01}+n_{10}n_{01}+m_{11}n_{11}}
\cP_{m_{10}}\cP_{m_{01}}\cP_{m_{11}}\cP_{n_{10}}\cP_{n_{01}}\cP_{n_{11}}.
\end{multline}
Now the two sides are ``just equal to each other''; they are not equal to a common simple expression, like $\cP_{\gamma_1}\cP_{\gamma_2}\cP_{\gamma_3}\cP_{\gamma_4}$ as in Equation \eqref{eqn:55term}. Although this new version does not seem to easily follow from any version of Reineke's identities, it translates naturally to an identity among quantum dilogarithm series
\[
\E(y_2)\E(y_3)\E(y_{12})\E(y_{34})\E(y_1)\E(y_4)
	= \E(y_1)\E(y_4)\E(y_{13})\E(y_{24})\E(y_2)\E(y_3),
\]
see Section \ref{s:mt} for notation.

\smallskip

The goal of this paper is to present a topological proof of Keller's dilogarithm identities associated with a pair of Dynkin quivers of type $A$, as follows.

\begin{thm*}[c.f.\ Theorem \ref{thm:mt}]
The following identity of quantum dilogarithm series holds in the completed quantum algebra of $A_n\square A_{n^\prime}$
\[
\prod_{(i,\phi)\in\Delta(A_{n'})\times\Phi(A_n)}^{\rightarrow} \E(y_{i,\phi}) = \prod_{(j,\psi)\in\Delta(A_n)\times\Phi(A_{n^\prime})}^{\rightarrow} \E(y_{j,\psi})
\]
where the products are respectively indexed by the simple and positive roots for the root systems corresponding to type $A_n$ and $A_{n^\prime}$ Dynkin diagrams. The arrows atop the products indicate that the products must be performed in a specific order.
\end{thm*}

The common value of the left-hand and right-hand sides above is called the Donaldson--Thomas invariant of the quiver with potential, denoted by $\E_{Q,W}$ \cite{bk2011}.

The main object in our proof is the representation space (a vector space) $\M_\gamma$ of the quiver $A_n \square A_{n'}$ acted upon by a group $\G_\gamma$.  We will consider two stratifications of the space $\M_\gamma$ and calculate the Poincar\'e series of $\coho^*_{\G_\gamma}(\M_\gamma)$ in two different ways corresponding to the two stratifications. The two Poincar\'e series expressions are hence equal, and also equal to the Poincar\'e series of $\coho^*(B\G_\gamma)$.  In fact, if we do what we just described we obtain identities of the type \eqref{eqn:55term} that are obvious consequences of Reineke's quantum dilogarithm identities. To achieve Keller's identities one further twist is needed, namely replacing ordinary equivariant cohomology with so-called {\em rapid decay cohomology}, introduced by Kontsevich--Soibelman, associated with a function (called superpotential trace) $W_\gamma: \M_\gamma \to \C$. With this twist, one obtains two explicit expressions corresponding to the two stratifications for the Poincar\'e series of the rapid decay cohomology algebra $\coho^*_{\G_\gamma}(\M_\gamma;W_\gamma)$. The equality in Theorem~\ref{thm:mt} translates to the fact that the two expressions are equal for any $\gamma$.

Calculating rapid decay equivariant cohomology algebras is rather difficult in general---it involves arguments over $\R$ even if our representation is complex algebraic. For example, we know no {\em direct} expression for the Poincar\'e series of $\coho^*_{\G_\gamma}(\M_\gamma;W_\gamma)$, only the ones obtained through the two stratifications mentioned above. Hence, let us comment on a few key points making our calculations work. One concerns the $G$-equivariant cohomology of a $G$-equivariant space $\eta$. It is well known that if $\eta$ is an orbit and $G_\nu$ is the stabilizer of an element $\nu\in \eta$, then $\coho^*_G(\eta)=\coho^*_{G_\nu}(\nu)$. One of our key steps is that this ``reduction to the normal form'' argument generalizes to certain situations (see Lemma \ref{lem:meets.every.orbit}) where $\eta$ is a family of orbits and $\nu$ an appropriate subset of $\eta$. The strata in our two stratifications of $\M_\gamma$ will have subsets $\nu$ for which the reduction argument works (see our Proposition \ref{prop:equivariant.eta.BG.eta}). Another main point of the proof is calculating the ``superpotential contributions''---c.f.\ the extra $q^{m_{11}n_{11}}$ factors distinguishing \eqref{eqn:55termKeller} from \eqref{eqn:55term}. On the one hand, detailed analysis of a quantum algebra will provide these contributions (see Section \ref{s:count.qalg}); on the other hand these contributions turn out to be related with the signature of a Hermitian form associated with our superpotential (Section~\ref{ss:hmtpy.norm.loci}).

For completeness let us mention that there is a formal difference between our main theorem (Theorem \ref{thm:mt}) and Keller's theorem \cite[Theorem~5.16 and Proposition~5.17]{bk2011}. Namely, in Keller's version one carries out a mutation algorithm on a graph for which the input consists of the quiver and the combinatorial data of a so-called \emph{maximal green sequence}; the two sides of the identities are described by the end positions of this algorithm. In other words, the resulting identity is not explicit in the sense that both sides are obtained by carrying out algorithms, requiring prior knowledge of two distinct maximal green sequences. In our version both sides of the identities are explicitly described. Nevertheless, we have no doubt that the two versions are the same.

\subsection*{Organization of the paper}
In Section \ref{s:quiver.prelim} we set out notations and give the needed background on quiver representations (especially for Dynkin quivers of type A), quivers with potential, and the quantum algebra associated to a quiver. In Section \ref{s:q.dilog.series} we define the quantum dilogarithm series which appear in our computation of the Donaldson--Thomas invariant. Section \ref{s:RDC} introduces rapid decay cohomology. In Section \ref{s:order.roots} we describe the construction of the square product and set out the notation used throughout the rest of the paper; in particular in the statement of our main theorem in Section \ref{s:mt}.

In Sections \ref{s:stratify.repspace}, \ref{s:w.strata}, and \ref{s:Kaz.Spec.Seq} we investigate the equivariant topology and geometry of the representation space which is necessary for our method. Most notably, in Section \ref{s:w.strata} we describe an explicit equivariant homotopy performed on the strata defined in Section \ref{s:stratify.repspace} which allows for the calculation of rapid decay cohomology. Furthermore, in Section \ref{ss:rdc.general}, we discuss the aspects of our method which would be necessary for a general framework to treat rapid decay cohomology algebras of general group representations, i.e.~not necessarily coming from representation spaces for (square products of) Dynkin quivers. Section \ref{s:count.qalg} describes the connection between certain combinatorial invariants from the geometry of our stratifications and arithmetic in the quantum algebra of the quiver. Finally, in Section \ref{s:pmt} we prove the main Theorem \ref{thm:mt}.

\subsection*{Acknowledgements} 

The idea of proving quantum dilogarithm identities by counting dimensions in a spectral sequence is due to Kontsevich and Soibelman; in fact, the present paper is guided by the arguments in \cite[Section 5]{mkys2011}, as well as the example on calculating rapid decay cohomology for the loop quiver with potential in \cite[Sect. 4.7]{mkys2011}. The authors also thank M.~Reineke and A.~Szenes for useful discussions on the topic of this paper. The first author acknowledges grant support from the Naval Academy Research Council and Office of Naval Research; the second author acknowledges the support of Simons Foundation grant 52388.

\section{Quiver preliminaries}
\label{s:quiver.prelim}

\subsection{Quiver representations}
\label{ss:quiv.rep}

A \defin{quiver} $Q = (Q_0,Q_1)$ is a directed graph with set of vertices $Q_0$ and set of directed edges $Q_1$ called \defin{arrows}. The maps $h:Q_1 \to Q_0$ and $t:Q_1\to Q_0$ respectively assign to each arrow its \defin{head} and \defin{tail}.

We will use the example on Figure \ref{fig:S.defn} as a running example throughout the paper, and call this quiver $S=A_2\square A_2$.
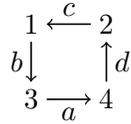
\begin{figure}
\begin{tikzpicture}[->,semithick,auto,inner sep=1mm]
\node (1) at (0,1) {$1$};
\node (2) at (1,1) {$2$};
\node (4) at (1,0) {$4$};
\node (3) at (0,0) {$3$};
\path
(1) edge node[left]  {$b$} (3)
(3) edge node[below] {$a$} (4)
(4) edge node[right] {$d$} (2)
(2) edge node[above] {$c$} (1);
\end{tikzpicture}
\caption{The quiver $S=A_2\square A_2$ with vertex set $S_0 = \{1,2,3,4\}$ and arrow set $S_1=\{a,b,c,d\}$ will serve as a running example throughout the paper.}
\label{fig:S.defn}
\end{figure}
For $S$ we have $h(a)=4$, $t(a)=3$, $h(d)=2$, \textit{et cetera}.

A vertex is called a \defin{source} (respectively \defin{sink}) if it is the tail (respectively head) of every arrow incident to it. A function $\gamma:Q_0\to\N$, or a choice of non-negative integer at each vertex, is called a \defin{dimension vector}. We will write $\gamma = (\gamma(v))_{v\in Q_0}$ in terms of its component functions. Throughout the rest of the paper, we let $\epsilon_v$ denote the unit dimension vector with a $1$ corresponding to vertex $v\in Q_0$ and zeroes elsewhere. For each choice of dimension vector $\gamma$ we define the \defin{representation space} of the quiver
\begin{equation}\label{eqn:M.defn}
\M_\gamma = \Dirsum_{a\in Q_1} \Hom\left(\C^{\gamma({t(a)})},\C^{\gamma(h(a))}\right)
\end{equation}
Elements of $\M_\gamma$ are called \defin{quiver representations}. The group $\G_\gamma = \prod_{v\in Q_0} \GL(\gamma(v),\C)$ acts on $\M_\gamma$ by base-change at each vertex; that is via
\begin{equation}\label{eqn:G.acts.M}
(g_v)_{v\in Q_0} \cdot (\phi_a)_{a\in Q_1} = (g_{h(a)} \phi_a g_{t(a)}^{-1})_{a\in \phi}.
\end{equation}
Each quiver comes equipped with a bilinear \defin{Euler form} $\chi :\N^{Q_0} \times \N^{Q_0} \to \Z$ which assigns to a pair of dimension vectors $\gamma_1 = (\gamma_{1}(v))_{v\in Q_0}$ and $\gamma_2=(\gamma_{2}(v))_{v\in Q_0}$ the integer
\begin{equation}\label{eqn:Euler.defn}
\chi (\gamma_1,\gamma_2 ) = \sum_{v\in Q_0} \gamma_{1}(v)\gamma_{2}(v) - \sum_{a\in Q_1} \gamma_{1}(t(a))\gamma_{2}(h(a)).
\end{equation}
Let $\lambda$ denote the opposite antisymmetrization of the Euler form
\begin{equation}\label{eqn:lambda.defn}
\lambda(\gamma_1,\gamma_2) = \chi( \gamma_2,\gamma_1 ) - \chi( \gamma_1,\gamma_2 ).
\end{equation}
In terms of the directed graph, $\lambda( \epsilon_u,\epsilon_v )$ reports the number of arrows $u\to v$ \emph{minus} the number of arrows $v\to u$. Hence, in the case that the quiver has no loops or double edges one has
\begin{equation}\label{eqn:lambda}
\lambda(\epsilon_u,\epsilon_v) =
\begin{cases}
+1, &\text{~if there is an arrow~} u\to v \\
-1, &\text{~if there is an arrow~} v\to u \\
\phantom{+}0, &\text{~otherwise}
\end{cases}
\end{equation}
The quadratic form $T_Q(\gamma) = \chi( \gamma,\gamma )$ is called the \defin{Tits form} of the quiver $Q$ which depends only on the underlying non-oriented graph of $Q$. A dimension vector $\gamma$ for which $T_Q(\gamma) = 1$ is called a \defin{root}.

\subsection{The path algebra of a quiver}
\label{ss:path.alg}

For any quiver $Q$, the \defin{path algebra} $\C Q$ is the $\C$-algebra spanned by all paths in the quiver, including the empty path at each vertex $\{\psi_i\}_{i\in Q_0}$. Multiplication is by concatenation of paths (read as function composition) whenever this makes sense, otherwise the product is zero. Observe that $\C Q$ is a unital ring with identity ${1}_{\C Q} = \sum_{i\in Q_0}\psi_i$. For example, in $\C S$ we have $\psi_1 \cdot b = a\cdot c = a \cdot d = 0$, but $d\cdot a = da$ and $b \cdot \psi_1 = b = \psi_3 \cdot b$.

Now consider the category of left $\C Q$-modules, which we denote by $\cqmod$. An object $M$ in this category determines a dimension vector by
\[\ddim(M) = (\dim(\psi_i \cdot M))_{i\in Q_0}.
\]
Therefore the Euler form can be defined on pairs of right $\C Q$-modules by \[\chi( M,N ) = \chi( \ddim(M),\ddim(N) ).\] The vector $\ddim(M)$ is called the \defin{dimension} of $M$.

For any two modules $M$ and $N$, it is known that $\Ext^i_{\C Q}(M,N) = 0$ for $i\geq 2$ and so $\Ext^1_{\C Q}(M,N)$ will be denoted simply by $\Ext(M,N)$. We will also write $\Hom(M,N)$ for $\Ext^0(M,N)$, the set of $\C Q$-module maps from $M$ to $N$. Furthermore, using the so-called Ringel (or standard projective) resolution (see for example~\cite[Theorem 2.15 and Proposition~8.4]{rs2014}) one obtains that
\begin{equation}
\label{eqn:chi.hom.ext}
\chi( M,N ) = \dhom(M,N) - \dext(M,N)
\end{equation}
where $\dhom(M,N) = \dim\Hom(M,N)$ and $\dext(M,N) = \dim\Ext(M,N)$. The full subcategory of $\cqmod$ corresponding to modules $M$ of fixed dimension $\gamma$ is known to be equivalent to the category whose set of objects is $\M_\gamma$ and whose morphisms $f:(\phi_{1,a})_{a\in Q_1} \to (\phi_{2,a})_{a\in Q_1}$ are given by vectors of linear mappings $f=(f_i)_{i\in Q_0}$ with the property that $f_{h(a)}\compose\phi_{1,a} = \phi_{2,a}\compose f_{t(a)}$ for each $a\in Q_1$, see for example \cite{rs2014}. We will implicitly use the categorical equivalence of $\C Q$-modules and quiver representations throughout the sequel.

\subsection{The quantum algebra of a quiver}
\label{ss:q.algebra}

Let $q^{1/2}$ be a variable and denote its square by $q$. The \defin{quantum algebra} $\A_Q$ of the quiver $Q$ is the $\Q(q^{1/2})$-algebra generated by symbols $y_\gamma$ for each dimension vector $\gamma$ and satisfying the relations
\begin{equation}\label{eqn:qalg.rel}
y_{\gamma_1+\gamma_2} = -q^{-\frac{1}{2} \lambda(\gamma_1,\gamma_2)}y_{\gamma_1}y_{\gamma_2}.
\end{equation}
The symbols $y_\gamma$ for each ${\gamma\in\N^{Q_0}}$ form a basis of $\A_Q$ as a vector space. The elements $\{y_{\epsilon_v}\}_{v\in Q_0}$ generate $\A_Q$ as an algebra. Notice that the relation \eqref{eqn:qalg.rel} implies that
\begin{equation}\label{eqn:qalg.comm}
y_{\gamma_1}y_{\gamma_2} = q^{\lambda(\gamma_1,\gamma_2)} y_{\gamma_2}y_{\gamma_1}.
\end{equation}
We let $\hat{\A}_Q$ denote the \defin{completed quantum algebra} in which formal power series in the symbols $y_\gamma$ are allowed. That is, $\hat{\A}_Q$ is the quotient of $\Q(q^{1/2})\langle\langle\{y_\gamma\}_{\gamma\in\N^{Q_0}}\rangle\rangle$ modulo relations given by Equation \eqref{eqn:qalg.rel}.

\begin{example}\label{ex:S.qA.rel}
In $\A_S$ denote $y_{\epsilon_i}$ by $y_i$ and $y_{\epsilon_i+\epsilon_j}$ by $y_{ij}$. Using the observation of Equation \eqref{eqn:lambda} a little computation gives that
\begin{align*}
y_2 y_1 & = q\,y_1y_2 & y_3y_1 &=q^{-1} y_1y_3 & y_4y_1 & =  y_1 y_4 \\
y_3 y_2 & = y_2 y_3 & y_4y_2 &= q\, y_2y_4 & y_4y_3 &= q^{-1}y_3y_4.
\end{align*}
One also obtains that
\begin{align*}
y_{12} &= -q^{1/2}y_1y_2 & y_{13} &= -q^{-1/2}y_1y_3 & y_{14} &= - y_1y_4 \\
y_{23} &= - y_2y_3 & y_{24} &=-q^{1/2}y_2y_4 & y_{34} & =-q^{-1/2}y_3y_4.
\end{align*}
\end{example}

\subsection{Quivers with potential}
\label{ss:quiv.with.pot}

A \defin{quiver with potential} is a pair $(Q,W)$ where $Q$ is a quiver and $W$ is an element of the space $\C Q/[\C Q,\C Q]$. That is, a monomial of $W$ is a cyclic path, i.e.\ starts and ends at the same vertex, but is unique only up to cyclic permutation. The element $W$ is called a \defin{superpotential}. The systematic study of the representation theory of quivers with potential was initiated by Derksen--Weymann--Zelevinsky \cite{hdjwaz2008,hdjwaz2010}, and is now a large and active field in its own right. One important aspect is that the superpotential naturally defines a regular function on the representation space; we will describe this in more detail in Sections \ref{ss:rdc.w} and \ref{ss:hmtpy.norm.loci}.

\begin{example}
\label{ex:W.S}
For the quiver $S$, a superpotential (up to cyclic permutation and scaling) is a linear combination of the paths $(abcd)^\ell$ for some natural number $\ell$. For the purpose of this paper we choose $W = - abcd$ as the superpotential on $S$.
\end{example}

\subsection{Dynkin quivers}
\label{ss:dynkin.quivers}

A \defin{Dynkin quiver} is a quiver whose underlying non-directed graph is a simply-laced Dynkin diagram; that is, of type $A$, $D$, or $E$. Let $Q$ be a Dynkin quiver with arbitrarily oriented arrows with corresponding set of positive roots $\Phi$. The indecomposable objects of $\cqmod$ are in one-to-one correspondence with the elements of $\Phi$ \cite{pg1972} and we denote the indecomposable module corresponding to $\beta\in \Phi$ by $M_\beta$. Simple roots are in bijection with vertices and in the sequel we will not distinguish between a simple root and its dimension vector $\epsilon_v$.

A \defin{Kostant partition} \cite{bk1959} of $\gamma$ is a vector of non-negative integers $\bfs{m}=(m_\beta)_{\beta\in\Phi}$ such that \[\gamma = \sum_{\beta\in\Phi} m_\beta \, \ddim({M_\beta}).\] We indicate $\bfs{m}$ is a Kostant partition for the dimension vector $\gamma$ by writing $\bfs{m} \kp \gamma$. Observe that for a fixed $\gamma$ there are only finitely many Kostant partitions.
Moreover, two $\C Q$-modules $M$ and $N$ with $\ddim(M)=\gamma=\ddim(N)$ are isomorphic if and only if their corresponding quiver representations are in the same $\G_\gamma$-orbit of $\M_\gamma$, and hence there are only finitely many $\G_\gamma$-orbits \cite{pg1972}. In particular, the $\G_\gamma$-orbits are in one-to-one correspondence with the set of Kostant partitions. Given $\bfs m \kp \gamma$ for a Dynkin quiver $Q$, we denote the associated $\G_\gamma$-orbit by $\Omega_{\bfs m}(Q)$. If there is no ambiguity regarding the quiver we simply write $\Omega_{\bfs m}$.

\subsection{Type $A$ quiver orbits}
\label{ss:type.A}

Let $N$ be a positive integer and let $Q$ be an orientation of a type $A_{N}$ Dynkin diagram. Write $Q_0 = \{1,2,\ldots,N\}$, labeled so that $Q$ is an orientation of
\begin{equation}
\label{eqn:AN.Dynkin}
\begin{tikzpicture}[baseline=(current  bounding  box.center),thick,black]

	\node (1) at (0,0) {$1$};
	\node (2) at (1,0) {$2$};
	\node (d) at (2.2,0) {$\cdots$};
	\node (N) at (3.4,0) {$N$};
	
	\path
		(1) edge (2)
		(2) edge (d)
		(d) edge (N);
\end{tikzpicture}.
\end{equation}

\begin{remark}
The positive roots for type $A_N$ correspond to intervals $[k,l]\subseteq [1,N]$. In particular, for each positive root $\beta\in \Phi$, there exist unique $k$ and $l$ such that $\beta = \sum_{i=k}^{l} \epsilon_{i}$. We let $\beta = \beta_{kl}$ denote this root throughout the paper.
\end{remark}

Given a dimension vector $\gamma = (\gamma(1),\ldots,\gamma(N))$ and a Kostant partition $\bfs{m}=(m_{\beta}) \kp \gamma$ defining the orbit $\Omega_{\bfs m}\subset\M_{\gamma}$, one can draw a \defin{lace diagram} as follows.

Consider a grid with $N$ columns. Draw $\gamma(i)$ dots in column $i$, justified at the top. Draw $m_{\beta}$ distinct line segment paths which begin at an unused dot in column $k$, ending at an unused dot in column $l$, and passing through one unused dot in each column between $k$ and $l$. Observe that for $\beta=\epsilon_{i}$ a simple root, there is no drawing to do, so $m_{\epsilon_{i}}$ dots in the $i$-th column are left untouched by line segments (but are still considered ``used''). Since $\bfs{m}\kp \gamma$ all dots will be used in the end.

A lace diagram for an orbit encodes all the information regarding the rank of the maps along each arrow and also incidence information regarding how the images, kernels, and cokernels of each map interact at vertices. By appropriately permuting the dots in each column, which is equivalent to acting by elements of $\G_{\gamma}$ consisting of permutation matrices, one can always obtain a lace diagram with no crossings. The choice of non-crossing lace diagram can be made unique if we specify a total ordering on the positive roots. For this purpose we choose to list the simple roots $\beta_{ii}$ last, and for the non-simple roots $\{\beta_{ij}:i<j\}$ we choose the lexicographic order on the subscripts; that is,
\begin{equation}
\label{eqn:lex.order.roots}
\beta_{ik} < \beta_{jl} \iff (i<j) \text{~or~} (i=j \text{~and~} k<l)
\end{equation}
Finally, the resulting \defin{canonical lace diagram} is obtained by drawing the line segments for roots $\beta$ in order, beginning from the top row, and always using the highest unused dot in each column. The importance is that the canonical lace diagram can be interpreted as a distinguished element in the $\G_{\gamma}$-orbit $\Omega_{\bfs{m}}$ as follows.

For each arrow $a$, form a $h(a)\times t(a)$ matrix by putting a $1$ in the $(i,j)$ spot whenever there is a segment in the lace diagram connecting the $j$-th dot from the top of the target column to the $i$-th dot from the top of the source column. Set all other matrix entries to $0$. We will call this element the \defin{normal form} for $\Omega_{\bfs{m}}$ and denote it by $\nu_{\bfs{m}}=(\nu_{a})_{a\in Q_{1}}$. Of course, the normal form \emph{does} depend on the orientation of the arrows.

\begin{example}
\label{ex:norm.form}
Consider the quiver $1\rightarrow 2 \rightarrow 3 \leftarrow 4$ with dimension vector $\gamma = (5,5,5,4)$. Take the Kostant partition with $m_{\beta_{14}}=2$ and $m_{\beta}=1$ for $\beta\in \{\beta_{12},\beta_{13},\beta_{24},\beta_{11},\beta_{33},\beta_{44}\}$ and $m_{\beta}=0$ otherwise. The resulting canonical lace diagram is
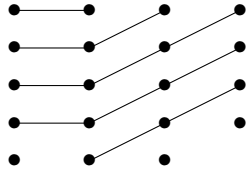
\begin{center}
	\begin{tikzpicture}
	\node (11) at (0,2) {$\bullet$};
	\node (12) at (0,1.5) {$\bullet$};
	\node (13) at (0,1) {$\bullet$};
	\node (14) at (0,.5) {$\bullet$};
	\node (15) at (0,0) {$\bullet$};
	\node (21) at (1,2) {$\bullet$};
	\node (22) at (1,1.5) {$\bullet$};
	\node (23) at (1,1) {$\bullet$};
	\node (24) at (1,.5) {$\bullet$};
	\node (25) at (1,0) {$\bullet$};
	\node (31) at (2,2) {$\bullet$};
	\node (32) at (2,1.5) {$\bullet$};
	\node (33) at (2,1) {$\bullet$};
	\node (34) at (2,.5) {$\bullet$};
	\node (35) at (2,0) {$\bullet$};
	\node (41) at (3,2) {$\bullet$};
	\node (42) at (3,1.5) {$\bullet$};
	\node (43) at (3,1) {$\bullet$};
	\node (44) at (3,.5) {$\bullet$};

	\draw (0,2) -- (1,2);
	\draw (0,1.5) -- (1,1.5) -- (2,2);
	\draw (0,1) -- (1,1) -- (3,2);
	\draw (0,.5) -- (1,.5) -- (3,1.5);
	\draw (1,0) -- (3,1);
	\end{tikzpicture}

\end{center}
which corresponds to the normal form
\[
\nu_{(m_{\beta})}=
\left(
\begin{bmatrix}
1 & 0 & 0 & 0 & 0 \\
0 & 1 & 0 & 0 & 0 \\
0 & 0 & 1 & 0 & 0 \\
0 & 0 & 0 & 1 & 0 \\
0 & 0 & 0 & 0 & 0
\end{bmatrix},
\begin{bmatrix}
0 & 1 & 0 & 0 & 0 \\
0 & 0 & 1 & 0 & 0 \\
0 & 0 & 0 & 1 & 0 \\
0 & 0 & 0 & 0 & 1 \\
0 & 0 & 0 & 0 & 0
\end{bmatrix},
\begin{bmatrix}
0 & 0 & 0 & 0 \\
1 & 0 & 0 & 0 \\
0 & 1 & 0 & 0 \\
0 & 0 & 1 & 0 \\
0 & 0 & 0 & 0
\end{bmatrix}
\right)
\]
in the representation space \[\M_{(5,5,5,4)}=\Hom(\C^{5},\C^{5})\dirsum\Hom(\C^{5},\C^{5})\dirsum\Hom(\C^{4},\C^{5}).\]
\end{example}

We introduce an operation on positive roots which we will use in Section \ref{ss:w.contrib.shifts}.

\begin{defn}
\label{defn:pos.rts.intersect}
Given two positive roots $\alpha' = \beta_{ij}$ and $\alpha'' = \beta_{uv}$, we let $\alpha'\intersect\alpha''$ denote the positive root $\beta_{st}$ where $s$ and $t$ are obtained by the intersection $[i,j]\intersect[u,v] = [s,t]$. Let $\delta(\alpha',\alpha'') := t-s$ denote the length of this interval. If the intersection of intervals above is empty, then we write $\alpha'\intersect\alpha'' = \emptyset$ and $\delta(\alpha',\alpha'') = 0$. When $[i,j]\intersect[u,v] = [s,t]$ is a non-empty intersection, we define $k(\alpha',\alpha'') = s$ and $\ell(\alpha',\alpha'') = t$. By abuse of notation, we will also write that $k(\alpha'') = k(\alpha'',\alpha'') = u$ and $\ell(\alpha'')=\ell(\alpha'',\alpha'')=v$.
\end{defn}

\section{Quantum dilogarithm series}
\label{s:q.dilog.series}

Given a variable $z$ we define the \defin{quantum dilogarithm series} to be the element in $\Q(q^{1/2})[[z]]$ defined by the series
\begin{equation}
\label{eqn:E.defn}
\E(z)= 1 + \sum_{j=1}^\infty \frac{(-z)^j q^{j^2/2}}{\prod_{k=1}^j (1-q^k)}.
\end{equation}

\begin{defn}
Given an $\R$-algebra $\curly{A}$, graded by the natural numbers, for which the $j$-graded piece $\curly{A}_j$ is a finite-dimensional $\R$-vector space for every $j\in\N$, the \defin{Poincar\'e series} of $\curly{A}$ in the variable $q^{1/2}$ is
\[
\cP[\curly{A}] = \sum_{j\in\N} q^{j/2} \dim_\R(A_j).
\]
\end{defn}

Throughout the sequel, we set $\cP_{j} = \cP[\coho^*(B\GL(j,\C))]$; this is the Poincar\'e series for equivariant cohomology of the group $\GL(j,\C)$. We recall that, since the cohomology $\coho^*(B\GL(j,\C))$ is a polynomial ring in the Chern classes of $\GL(j,\C)$, all of the odd cohomology groups vanish, and we have $\cP_j = \sum_{r\geq 0} q^r\,\dim(\coho^{2r}(B\GL(j,\C)))$, and obtain that $\curly P_0 = 1$ and $\cP_j = \prod_{r=1}^j (1-q^r)^{-1}$ for $j>0$. Thus we notice that the quantum dilogarithm series can be written as
\begin{equation}
\label{eqn:E.defn.poincare}
\E(z) = \sum_{j\geq 0} (-z)^j q^{j^2/2}\, \cP_{j}.
\end{equation}
We remark that often in the literature, see for example \cite{bk2011}, the quantum dilogarithm series is instead defined to be
\begin{equation}\label{eqn:KE.defn}
1 + \sum_{j\geq 1} \frac{z^j q^{j^2/2}}{(q^j-1)(q^j-q)\cdots(q^j-q^{j-1})}.
\end{equation}
Note that in the series \eqref{eqn:KE.defn}, the denominators count elements of the group $\GL(j,\mathbb{F}_q)$. The two formulations \eqref{eqn:E.defn} and \eqref{eqn:KE.defn} are images of each other under the involution $q^{1/2} \mapsto -q^{-1/2}$. However, it is more convenient for our purposes to count generators in $\coho^*(B\GL(j,\C)) \iso \coho^*(B\U(j))$ and therefore we choose to work with the formulation given by Equation \eqref{eqn:E.defn}. This is consistent with the conventions in \cite{rr2013}.

\section{Rapid decay cohomology}
\label{s:RDC}

\subsection{Definition}
\label{ss:RDC.defn}

Let $X$ be a complex algebraic variety and $f:X\to \C$ a regular function on $X$. Let $\Re[z]$ denote the real part of the complex number $z$, and for every $t\in \R$ define the set $S_t = \{z \in \C : \Re[z]\leq t\}$ and let $X_f(t) = f^{-1}(S_t) \subset X$. The \defin{rapid decay cohomology} $\coho^*(X;f)$ is the limit as $t\to -\infty$ of the relative cohomology of the pair $\coho^*(X,X_f(t))$. It is known that this cohomology stabilizes for finite $t_0\in\R$ with $t_0 \ll 0$. For more on this definition and the choice of terminology, see \cite[Section~4.1]{mkys2011}.

If the algebraic group $G$ acts on $X$ such that $X_f(t)$ is invariant for all $t\ll0$ then the \defin{equivariant rapid decay cohomology} is also defined; explicitly $\coho^*_G(X;f) = \lim_{t \to -\infty} \coho^*_G(X,X_f(t))$. In the sequel we will use the equivariant version but simply say ``rapid decay cohomology'' and omit the extra adjective.

\subsection{Rapid decay cohomology from superpotentials}
\label{ss:rdc.w}

For any quiver with potential $(Q,W)$, fix a dimension vector $\gamma$. We obtain a regular function $W_\gamma:\M_\gamma\to\C$ as follows.

Given a path $p=a_{1}a_{2}\ldots a_{\ell}$ in the quiver and an element $\phi=(\phi_{a})_{a\in Q_{1}}\in \M_{\gamma}$, let $\phi_{p}$ denote the composition $\phi_{a_{1}}\phi_{a_2}\cdots\phi_{a_\ell}$. If $p$ forms an oriented cycle then it makes sense to consider $\tr(\phi_p)$. For $W = \sum_p c_p p$ a finite linear combination of oriented cycles $p$, we set
\begin{equation}
\label{eqn:W.gamma.defn}
W_\gamma(\phi) = \sum_p c_p \tr(\phi_p),
\end{equation}
which is well-defined because $\tr(\phi_{a_{1}}\phi_{a_2}\cdots\phi_{a_\ell})$ is invariant under cyclic permutations of the $\phi_{a_i}$ factors.

\begin{example}
\label{ex:trace.function.S}
For $S$ with superpotential $W = -abcd$ we obtain $W_\gamma:\M_\gamma \to \C$
\begin{equation}
\label{eqn:Wgamma.S}
\phi = (\phi_a,\phi_b,\phi_c,\phi_d) \longmapsto
-\tr(\phi_a\compose\phi_b\compose\phi_c\compose\phi_d).
\end{equation}
Note that for any $g=(g_1,g_2,g_3,g_4)\in \G_\gamma$ we have
\[
\begin{aligned}
W_\gamma(g\cdot\phi)
	&= -\tr\left((g_4\phi_a g_3^{-1}) (g_3 \phi_b g_1^{-1}) (g_1 \phi_c g_2^{-1}) (g_2 \phi_d g_4^{-1})\right) \\
	&= -\tr(g_4 \phi_a\phi_b\phi_c\phi_d g_4^{-1}) = W_\gamma(\phi)
\end{aligned}
\]
where the last equality follows because trace is an invariant of conjugacy classes. Therefore the rapid decay cohomology $\coho^*_{\G_\gamma}(\M_\gamma;W_\gamma)$ is well-defined. We remark that a similar proof shows that for any quiver with potential $(Q,W)$ and $W_\gamma$ determined by Equation \eqref{eqn:W.gamma.defn}, the equivariant rapid decay cohomology $\coho^*_{\G_\gamma}(\M_\gamma;W_\gamma)$ is well-defined.
\end{example}

\section{Notation and ordering of roots in square products}
\label{s:order.roots}

\subsection{The quiver $Q = A_n\square A_{n^\prime}$}
\label{ss:not.Q}
The following construction of the square product of $A_n$ and $A_{n^\prime}$ follows that of Keller  \cite{bk2011,bk2013.fpsac}. Label the vertices of $A_n$ and $A_{n^\prime}$ with elements of $[n]$ and $[n^\prime]$ as in \eqref{eqn:AN.Dynkin}. In this way, $Q_0$ is identified with $\{(i,j)\,:\,1\leq i \leq n,\,1\leq j \leq n^\prime\}$. We will use standard matrix notation for these vertices, realizing $Q_0$ on a square grid. Assign alternating orientations to $A_n$ and $A_{n^\prime}$ such that the vertex labeled $1$ is a source in $A_n$ and a sink in $A_{n^\prime}$. Denote these oriented quivers by $\vec{A}_n$ and $\vec{A}_{n^\prime}$ respectively. Form $Q$ by reversing the arrows in the full subquiver $\{i\}\times \vec{A}_{n^\prime}$ and $\vec{A}_n \times \{j\}$ whenever $i$ is a sink in $\vec{A}_n$ and $j$ is a source in $\vec{A}_{n^\prime}$. This produces a grid of oriented squares. Figure \ref{fig:A3.A4.label} depicts the quiver $A_3\square A_4$ with this labeling system applied.

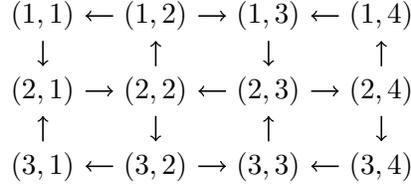
\begin{figure}
\begin{tikzpicture}[->,semithick,auto]

\node (11) at (0,2) {$(1,1)$};
\node (21) at (0,1) {$(2,1)$};
\node (31) at (0,0) {$(3,1)$};
\node (12) at (1.5,2) {$(1,2)$};
\node (22) at (1.5,1) {$(2,2)$};
\node (32) at (1.5,0) {$(3,2)$};
\node (13) at (3,2) {$(1,3)$};
\node (23) at (3,1) {$(2,3)$};
\node (33) at (3,0) {$(3,3)$};
\node (14) at (4.5,2) {$(1,4)$};
\node (24) at (4.5,1) {$(2,4)$};
\node (34) at (4.5,0) {$(3,4)$};

\path
(11) edge  (21)
(31) edge  (21)
(22) edge  (12)
(22) edge  (32)
(13) edge  (23)
(33) edge  (23)
(24) edge  (14)
(24) edge  (34)
(12) edge  (11)
(12) edge  (13)
(14) edge  (13)
(21) edge  (22)
(23) edge  (22)
(23) edge  (24)
(32) edge  (31)
(32) edge  (33)
(34) edge  (33);
\end{tikzpicture}
\caption{The quiver $A_3\square A_4$ where, for example, $Q(\bullet,2)$ is the ``outbound $A_3$" sub-quiver $(1,2)\leftarrow(2,2)\rightarrow(3,2)$.}
\label{fig:A3.A4.label}
\end{figure}

Observe that to each vertex $i$ of $A_n$ there is an associated alternating $A_{n^\prime}$ quiver (across a row) which we denote by $\rowQ{i}$ and similarly for each vertex $j$ of $A_{n^\prime}$ there is an associated alternating $A_n$ quiver (down a column) which we denote by $\colQ{j}$. We call $\rowQ{i}$ a \defin{horizontal sub-quiver} and $\colQ{j}$ a \defin{vertical sub-quiver}. These conventions imply that if $(i,j)$ is a sink (respectively source) in $\rowQ{i}$ then it is a source (resp.~sink) in $\colQ{j}$.

A vertex $(i,j)$ is called a \defin{horizontal head} (respectively a \defin{horizontal tail}) if it is a sink (resp.~a source) in the horizontal sub-quiver $\rowQ{i}$; that is, if it is a source (resp.~sink) when one considers only the horizontal arrows as depicted in the example of Figure \ref{fig:A3.A4.label}. We define \defin{vertical head} and \defin{vertical tail} similarly. We denote the respective sets of horizontal heads and tails by $\hhs$ and $\hts$, and vertical heads and tails by $\vhs$ and $\vts$. Observe that $\hhs = \vts$ and similarly $\vhs = \hts$. The two sets $\hhs = \vts$ and $\hts = \vhs$ are respectively called \defin{even} and \defin{odd} vertices in Keller's description of square product quivers \cite{bk2013.fpsac}.

Let $\epsilon_{i,j} = \epsilon_{(i,j)}$ denote the \defin{simple root} associated to the vertex $(i,j)$ which, we recall, is identified with the dimension vector having a $1$ in the spot corresponding to the vertex $(i,j)$ and zeroes elsewhere. Observe $T_Q(\epsilon_{i,j}) = 1$ so that it is indeed a root of $Q$ and moreover, corresponds to a simple root in the sense of root systems for simple Lie algebras for each of the sub-quivers $\rowQ{i}$ and $\colQ{j}$. Let
\[
\hpri{i} = \{\beta_{k,l}^{(i)}\,:\,1\leq k\leq l\leq n^\prime\}
\]
denote the positive roots for $\rowQ{i}$ where
\[
\textstyle\beta_{k,l}^{(i)} = \sum_{j=k}^l \epsilon_{i,j}.
\]
It is not difficult to check that each $\beta_{k,l}^{i}$ is also a root of $Q$; i.e.~$T_Q(\beta_{k,l}^{i}) = 1$, but we will not explicitly need this fact in the sequel. Similarly we define
\[
\vprj{j} = \{\beta^{(j)}_{k,l} = {\textstyle\sum_{i=k}^l \epsilon_{i,j}}\}
\]
as the positive roots for the vertical sub-quiver $\colQ{j}$.

Let $\beta'\in\hpri{i'}$ and $\beta''\in\hpri{i''}$. We say that $\beta'$ and $\beta''$ are from the \defin{same row in $Q$} if $i'=i''$. We define the \defin{same column in $Q$} analogously. Set
\begin{align*}
\hpr &= \Union_{i\in[n]} \hpri{i} &\text{and} &
&\vpr &= \Union_{j\in[n^\prime]} \vprj{j}
\end{align*}
which we call \defin{horizontal positive roots} and \defin{vertical positive roots} respectively.

We endow $Q$ with the structure of a quiver with potential as follows. Let $p_{ij}$ denote the oriented cycle involving the vertices $(i,j)$, $(i,j+1)$, $(i+1,j+1)$, and $(i+1,j)$. This is a path in $Q$ which starts and ends at the same vertex. So, we may define a superpotential for $Q$ by
\begin{equation}
\label{eqn:W.for.Q}
W = - \sum_{i=1}^{n-1}\sum_{j=1}^{n^\prime-1} p_{ij}
\end{equation}
which we fix throughout the rest of the paper.

Finally, given a dimension vector $\gamma = (\gamma(v))_{v\in Q_0}$, we denote the dimension of the vector space at vertex $(i,j)$ not by $\gamma((i,j))$, but more simply by $\gamma(i,j)$.

\subsection{Ordering positive roots}
\label{ss:ord.rts}

We define an ordering $(\hpr,\prec)$ on the set of horizontal positive roots by the following rules. Choose distinct $\beta',\beta'' \in \hpr$. If $\beta'$ and $\beta''$ are from the {same row} of $Q$ then
\begin{subequations}
\begin{align}
\beta'\prec\beta'' &\implies \lambda(\beta',\beta'')\geq 0.
\label{eqn:row.ord.within}
\intertext{%
On the other hand, if $\beta'$ and $\beta''$ are from {different rows} of $Q$, then
}
\beta'\prec\beta'' &\implies \lambda(\beta',\beta'')\leq 0.
\label{eqn:row.ord.without}
\end{align}
\end{subequations}
Because $\lambda(\beta',\beta'')=0$ for many choices of $\beta'$ and $\beta''$, this ordering is not unique. However, we remark that Equation \eqref{eqn:qalg.comm} implies that if $\lambda(\beta',\beta'')=0$ then the elements $y_{\beta'}$ and $y_{\beta''}$ commute in the quantum algebra $\A_Q$ and its completion $\hat\A_Q$. This is on purpose---our order is chosen to exploit commutation relations in the quantum algebra---a connection we explain in Sections \ref{s:count.qalg} and \ref{s:pmt}.

Similarly, if $\beta',\beta'' \in \vpr$ are from the {same column} then
\[
\beta'\prec\beta'' \implies \lambda(\beta',\beta'')\geq 0,
\]
and if $\beta',\beta'' \in \vpr$ are from {different columns} then
\[
\beta'\prec\beta'' \implies \lambda(\beta',\beta'')\leq 0.
\]
The second author defined an order on the positive roots for any Dynkin quiver in \cite[Section~4]{rr2013} as follows. For positive roots $\beta',\beta''$ from the same row, i.e.\ $\beta'$ and $\beta''$ are both positive roots for the same Dynkin quiver $\rowQ{i}$, we have
\begin{equation}
\label{eqn:rr.pos.roots.order}
\beta' \prec \beta'' \implies
	\Hom(M_{\beta'},M_{\beta''}) = 0
	\text{~and~}
	\Ext(M_{\beta''},M_{\beta'}) = 0
\end{equation}
where $M_\alpha$ denotes the indecomposable $\C\rowQ{i}$-module corresponding to the positive root $\alpha$.

\begin{lem}
\label{lem:rr.order.is.our.order}
The ordering of positive roots within the same row prescribed by \eqref{eqn:row.ord.within} is equivalent to the ordering prescribed by \eqref{eqn:rr.pos.roots.order} in the following sense:
\begin{enumerate}[leftmargin=*,label=(\alph*)]
\item if $\alpha\prec\beta$ is allowed by \eqref{eqn:row.ord.within} and $\beta\prec\alpha$ is allowed by \eqref{eqn:rr.pos.roots.order}, then $\lambda(\alpha,\beta)=0$;
\item if $\alpha\prec\beta$ is allowed by \eqref{eqn:rr.pos.roots.order} and $\beta\prec\alpha$ is allowed by \eqref{eqn:row.ord.within}, then $\lambda(\alpha,\beta)=0$.
\end{enumerate}
The analogous statements hold for vertical sub-quivers and vertical positive roots.
\end{lem}

\begin{proof}
(a) Suppose that $\alpha\prec\beta$ is allowed by \eqref{eqn:row.ord.within}. Then $\lambda(\alpha,\beta) = \chi(\beta,\alpha) - \chi(\alpha,\beta) \geq 0$ which, by Equation \eqref{eqn:chi.hom.ext}, implies that
\[
\dhom(M_\beta,M_\alpha)+\dext(M_\alpha,M_\beta) \geq  \dhom(M_\alpha,M_\beta) + \dext(M_\beta,M_\alpha).
\]
Now if $\beta\prec\alpha$ is allowed by \eqref{eqn:rr.pos.roots.order}, then the left hand side above vanishes. Since the right-hand side is non-negative, it must also vanish, and hence $\lambda(\alpha,\beta) = 0$.

(b) Suppose now that $\alpha \prec \beta$ is allowed by \eqref{eqn:rr.pos.roots.order}. Then using Equation \eqref{eqn:chi.hom.ext} again, we obtain that
\begin{align*}
\lambda(\alpha,\beta) =
	&\,\chi(M_\beta,M_\alpha) - \chi(M_\alpha,M_\beta) \\	
=	&\left[\dhom(M_\beta,M_\alpha) - \dext(M_\beta,M_\alpha) \right] \\
	&\quad - \left[\dhom(M_\alpha,M_\beta) - \dext(M_\alpha,M_\beta) \right] \\
=	&\dhom(M_\beta,M_\alpha) + \dext(M_\alpha,M_\beta)
\end{align*}
which is non-negative, i.e.\ $\lambda(\alpha,\beta)\geq 0$. On the other hand, if $\beta\prec\alpha$ is allowed by \eqref{eqn:row.ord.within} then $\lambda(\beta,\alpha) \geq 0$ which establishes the claim.
\end{proof}

\subsection{On the existence of our ordering}
\label{ss:order.exist}

Observe that when $k$ is odd, the subquiver $\rowQ{k}$ has the form
\[
(k,1) \leftarrow (k,2) \rightarrow (k,3) \leftarrow (k,4) \;\cdots\; (k,n').
\]
We organize the roots $\hpri{k}$ into an $(n'+1)\times n'$ matrix $\curly{M}^{(k)}$ as follows. In the $(i,j)$-entry $\curly{M}_{ij}^{(k)}$ with $i+j$ \emph{odd}, we put $\beta^{(k)}_{u,v}$ where
\begin{equation}
\label{eqn:order.matrix}
u=\begin{cases}
j-i+1 & \text{if~} j\geq i \\
i-j & \text{if~} j<i
\end{cases}
\quad and \quad
v=\begin{cases}
i+j-1 & \text{if~} i+j \leq n'+1 \\
2n'+2-i-j & \text{if~} i+j >n'+1
\end{cases}
\end{equation}
and leave the other entries blank. Alternatively when $k$ is even, $\rowQ{k}$ has the form
\[
(k,1) \rightarrow (k,2) \leftarrow (k,3) \rightarrow (k,4) \;\cdots\; (k,n').
\]
and we put $\beta^{(k)}_{u,v}$ in the $(i,j)$-entry $\curly{M}^{(k)}_{ij}$ with $i+j$ \emph{even}, again with $u$ and $v$ defined as in \eqref{eqn:order.matrix}.

\begin{defn}
\label{defn:row.number}
Let $\rho(\beta^{(k)}_{u,v})$ denote the row index of the entry $\beta^{(k)}_{u,v}$ in the matrix $\curly{M}^{(k)}$.
\end{defn}

\begin{example}
\label{ex:order.matrix.rho}
Let $n'=5$ and $k=2$. We have
\[
\curly{M}^{(2)} =
\left(
\begin{array}{ccccc}
 \beta^{(2)}_{1,1} &   & \beta^{(2)}_{3,3} &   & \beta^{(2)}_{5,5} \\
   & \beta^{(2)}_{1,3} &   & \beta^{(2)}_{3,5} &   \\
 \beta^{(2)}_{2,3} &   & \beta^{(2)}_{1,5} &   & \beta^{(2)}_{3,4} \\
   & \beta^{(2)}_{2,5} &   & \beta^{(2)}_{1,4} &   \\
 \beta^{(2)}_{4,5} &   & \beta^{(2)}_{2,4} &   & \beta^{(2)}_{1,2} \\
   & \beta^{(2)}_{4,4} &   & \beta^{(2)}_{2,2} &   \\
\end{array}
\right)
\]
and notice that $\rho(\beta^{(2)}_{1,5}) = 3$ and $\rho(\beta^{(2)}_{1,3}) = 2$.
\end{example}

\begin{remark}
\label{rem:ord.mat.AR.position}
Observe that the arrangement of the roots $\beta^{(k)}_{u,v}$ into the $\curly{M}^{(k)}$ tables is the natural position of them (or more precisely their corresponding indecomposable $\C\rowQ{k}$-modules) when one builds the Auslander--Reiten graph of an alternating $A$-type quiver using the ``knitting algorithm'', see e.g.~\cite[Chapter~3]{rs2014}.
\end{remark}

\begin{thm}
\label{thm:order.exist}
With the definitions above, we have
\begin{enumerate}[label={$\bullet$\;\textbf{\emph{Claim~\arabic*.}}},leftmargin=*,align=left]
\item if $\rho(\beta^{(k)}_{u,v}) = \rho(\beta^{(l)}_{s,t})$, then $\lambda(\beta^{(k)}_{u,v},\beta^{(l)}_{s,t}) = 0$;
\item if $\rho(\beta^{(k)}_{u,v}) < \rho(\beta^{(l)}_{s,t})$ and $k\neq l$, then $\lambda(\beta^{(k)}_{u,v},\beta^{(l)}_{s,t}) \leq 0$;
\item if $\rho(\beta^{(k)}_{u,v}) < \rho(\beta^{(l)}_{s,t})$ and $k = l$, then $\lambda(\beta^{(k)}_{u,v},\beta^{(l)}_{s,t}) \geq 0$.
\end{enumerate}
\end{thm}

\begin{proof}
By the nature of defining $\rho$, each of these claims breaks down to a finite number of possible cases. For each of the possibilities, the proof is a straightforward combinatorial verification. For illustration, we show one case.

Let $k=1$ and $l=2$ (see Remark \ref{rem:|k-l|>1}). Choose even integers $\rho_1 < \rho_2$ and choose roots $\beta^{(1)}_{u,v}$ and $\beta^{(2)}_{s,t}$ such that $\rho_1 = \rho(\beta^{(1)}_{u,v})$ and $\rho_2 = \rho(\beta^{(2)}_{s,t})$; that is, we will verify a case of Claim 2. Let $j_1$ and $j_2$ denote the column indices respectively of $\beta^{(1)}_{u,v}$ in $\curly{M}^1$ and $\beta^{(2)}_{s,t}$ in $\curly{M}^2$. Observe that since $\rho_1 + j_1$ and $\rho_2 + j_2$ are respectively odd and even, we must have that $j_1$ is odd and $j_2$ is even.

Further suppose that we are in the case that $j_1 \geq \rho_1$, $j_1 + \rho_1 \leq n'+1$, $j_2 \geq \rho_2$, and $j_2 + \rho_2 > n' +1$. This implies that
\begin{align*}
u & = j_1 - \rho_1 + 1 & v & = \rho_1 + j_1 - 1 \\
s & = j_2 - \rho_2 + 1 & t & = 2n'+2-\rho_2-j_2
\end{align*}
so we see that $s$ is odd, while $u$, $v$, and $t$ are even. To compute the value of $\lambda(\beta^{(1)}_{u,v},\beta^{(2)}_{s,t})$ we need to determine the \emph{overlap} of the intervals $[u,v]$ and $[s,t]$ and the number of up and down arrows from $Q$ appearing in this overlap.

If the overlap is empty, then $\lambda = 0$ and the claim is trivial. Otherwise, we know that the right-hand endpoint of the overlap must be even (since $v$ and $t$ are even), say $2f$ for some integer $f$. If $u>s$, then the left-hand endpoint is even and the portion of $Q$ in the overlap looks like:
\begin{center}
\begin{tikzpicture}
\node (11) at (0,1.5) {$(1,u)$};
\node (12) at (3,1.5) {$(1,u+1)$};
\node (13) at (6,1.5) {$(1,u+2)$};
\node (1m-1) at (9,1.5) {$(1,2f-1)$};
\node (1m) at (12,1.5) {$(1,2f)$};

\node (21) at (0,0) {$(2,u)$};
\node (22) at (3,0) {$(2,u+1)$};
\node (23) at (6,0) {$(2,u+1)$};
\node (2m-1) at (9,0) {$(2,2f-1)$};
\node (2m) at (12,0) {$(2,2f)$};

\node at (7.5,.75) {$\cdots$};

\draw[->] (11) -- (12);
\draw[->] (13) -- (12);
\draw[->] (1m) -- (1m-1);
\draw[->] (22) -- (21);
\draw[->] (22) -- (23);
\draw[->] (2m-1) -- (2m);
\draw[->] (21) -- (11);
\draw[->] (12) -- (22);
\draw[->] (23) -- (13);
\draw[->] (1m-1) -- (2m-1);
\draw[->] (2m) -- (1m);
\end{tikzpicture}
\end{center}
Hence the value of $\lambda(\beta^{(1)}_{u,v},\beta^{(2)}_{s,t}) = -1$ since there is $1$ more upward arrow than downward arrow above. Conversely, if $s>u$, then the left-hand endpoint of the overlap will be odd, and there will be exactly as many upward and downward arrows in the corresponding picture. In this case we obtain $\lambda(\beta^{(1)}_{u,v},\beta^{(2)}_{s,t}) = 0$.
\end{proof}

\begin{remark}
\label{rem:|k-l|>1}
Observe that if $|k-l|>1$ then the claims are trivial.
\end{remark}

\begin{remark}
\label{rem:k=1}
In the cases when $k=l$, the claims actually follow from Lemma \ref{lem:rr.order.is.our.order}, Remark \ref{rem:ord.mat.AR.position}, and the fact the dimensions of $\Hom$ and $\Ext$ between indecomposable $\C\rowQ{k}$-modules can be read off from the shape of the Auslander--Reiten quiver, see e.g.~\cite[Chapter~3]{rs2014}.
\end{remark}

\begin{defn}
Let $\hpr_r$ denote the set $\{\beta\in\hpr : \rho(\beta)=r\}$.
\end{defn}

\begin{cor}
\label{cor:order.exist}
The order obtained by listing horizontal positive roots in $\hpr_r$ in arbitrary order (for each $r$), and putting
\[
\hpr_1 \prec \hpr_2 \prec \cdots \prec \hpr_{n'+1}
\]
satisfies the requirements of Section \ref{ss:ord.rts}.
\end{cor}

\begin{proof}
Claim 1 implies that each $\hpr_r$ can be ordered arbitrarily. Claims 2 and 3 respectively guarantee that roots coming from different and same rows of $Q$ (in the sense of Section \ref{ss:not.Q}) are appropriately ordered.
\end{proof}

\section{The main theorem}
\label{s:mt}

Set $a=|\hpr| = nn^\prime(n^\prime+1)/2$ and $b=|\vpr|=n^\prime n(n+1)/2$ and choose labelings for the horizontal and vertical positive roots from the respective sets $[a]$ and $[b]$ which respect the orderings defined in Section \ref{ss:ord.rts}. That is, write $\hpr=\{\phi_r\,:\,r\in[a]\}$ and $\vpr=\{\psi_s\,:\,s\in[b]\}$ such that
\begin{subequations}
\begin{gather}
\phi_1\prec\phi_2\prec\cdots\prec\phi_a\label{eqn:hor.ord} \\
\psi_1\prec\psi_2\prec\cdots\prec\psi_b.\label{eqn:ver.ord}
\end{gather}
\end{subequations}

\begin{thm}\label{thm:mt}
With the orderings \eqref{eqn:hor.ord} and \eqref{eqn:ver.ord}, the following identity of quantum dilogarithm series holds in the completed quantum algebra $\hat\A_Q$
\begin{equation}
\label{eqn:mt}
\E(y_{\phi_1})\E(y_{\phi_2})\cdots\E(y_{\phi_a}) = \E(y_{\psi_1})\E(y_{\psi_2})\cdots\E(y_{\psi_b}).
\end{equation}
\end{thm}

\begin{defn}
\label{defn:DT}
The common value of both sides of Equation \eqref{eqn:mt} is called the \defin{Donaldson}--\defin{Thomas} \defin{invariant} (DT invariant) of the quiver with potential $(Q,W)$ \cite{bk2011, bk2013}.
\end{defn}

We will prove this theorem in Section \ref{s:pmt}.

\begin{example}\label{ex:}
Combining the notation $\epsilon_{i,j}$ from Section \ref{ss:not.Q} for simple roots and the convention for numbering the vertices of $S$ from Figure \ref{fig:S.defn}, we write
\begin{align*}
\zeta_1 &= \epsilon_{11} & \zeta_4 &=\epsilon_{22} & \zeta_2 &=\epsilon_{12} & \zeta_3 &=\epsilon_{21} \\
\zeta_{12} &=\zeta_1+\zeta_2 & \zeta_{34} &=\zeta_3 + \zeta_4 & \zeta_{13} &=\zeta_1 + \zeta_3 & \zeta_{24} &=\zeta_2 + \zeta_4.
\end{align*}
One can check that
\begin{align*}
\zeta_2 \prec \zeta_3 \prec \zeta_{12} \prec \zeta_{34} \prec \zeta_1 \prec \zeta_4
\intertext{is an allowed ordering of the horizontal roots and that}
\zeta_1 \prec \zeta_4 \prec \zeta_{13} \prec \zeta_{24} \prec \zeta_1 \prec \zeta_4
\end{align*}
is an allowed ordering of the vertical roots. Let $y_\bullet$ denote $y_{\zeta_\bullet}$. Then the main theorem asserts that
\begin{equation}
\label{eqn:dilog.S.2211}
\E(y_2)\E(y_3)\E(y_{12})\E(y_{34})\E(y_1)\E(y_4) 
	= \E(y_1)\E(y_4)\E(y_{13})\E(y_{24})\E(y_2)\E(y_3).
\end{equation}
We remark that the list of quantum dilogarithms above is grouped into commuting pairs. Namely, one can check that $\{y_{2},y_{3}\}$, $\{y_{12},y_{34}\}$, $\{y_{13},y_{24}\}$ and $\{y_{1},y_{4}\}$ are all commuting sets in $\hat\A_{Q}$.
\end{example}

\section{Stratification of the representation space}
\label{s:stratify.repspace}

Let $\gamma$ be a dimension vector for $Q$ and let $\gamma(i,\bullet)$, $\gamma(\bullet,j)$ denote the respective resulting dimension vectors on $Q(i,\bullet)$ and $Q(\bullet,j)$. We define a \defin{horizontal Kostant series of} $\gamma$ to be a list $\bfs{m}^\bullet = (\bfs{m}^1,\ldots,\bfs{m}^n)$ where each $\bfs{m}^i = (m^i_\beta)_{\beta\in\hpri{i}}$ is a Kostant partition of $\gamma(i,\bullet)$. Similarly we can define \defin{vertical Kostant series} $(\bfs{m}^1,\ldots,\bfs{m}^{n^\prime})$ with each $\bfs{m}^j=(m^j_\beta)_{\beta\in\vprj{j}}$ a Kostant partition for $\gamma(\bullet,j)$. In an abuse of notation, we will write $\bfs{m}^\bullet \vdash \gamma$. Equivalently, we will sometimes use the notation $\bfs{m}^\bullet = (m_\beta)_{\beta\in\hpr}$ for a horizontal Kostant series, and analogously $\bfs{m}^\bullet = (m_\beta)_{\beta\in\vpr}$ for vertical Kostant series. The intended notation will be clear from context. Observe that since there are only finitely many Kostant partitions for each $\gamma(i,\bullet)$ (respectively $\gamma(\bullet,j)$), there are only finitely many horizontal (resp.\ vertical) Kostant series for $\gamma$.

\begin{defn}
\label{defn:strata}
For every horizontal Kostant series $\bfs{m}^\bullet$ we define $\eta(\bfs{m}^\bullet)\subset\M_\gamma$ as the locus where the maps along horizontal arrows in each $Q(i,\bullet)$ belong to the type $A_{n^\prime}$ quiver orbit $\Omega_{\bfs{m}^i}(\rowQ{i})$. We allow complete freedom on the vertical arrows. We call $\eta(\bfs{m}^\bullet)$ a \defin{horizontal stratum}. Similarly we define a \defin{vertical stratum} $\theta(\bfs{m}^\bullet) \subset \M_\gamma$ associated to the vertical Kostant series $\bfs{m}^\bullet$. Let $\Hor(\M_\gamma)$ and $\Ver(\M_\gamma)$ denote respectively the sets of all horizontal and vertical strata in $\M_\gamma$. Often, when the Kostant series is understood, we will more succinctly write $\eta = \eta(\bfs{m}^\bullet)$ and $\theta = \theta(\bfs{m}^\bullet)$.
\end{defn}

Observe that the horizontal and vertical strata are $\G_\gamma$-invariant and
\[\Union_{\eta\in\Hor(\M_\gamma)} \eta = \M_\gamma = \Union_{\theta\in\Ver(\M_\gamma)} \theta.\]
The following proposition follows immediately from the definitions.

\begin{prop}
\label{prop:codim.eta}
The (complex) codimension of $\eta(\bfs{m}^\bullet) \subset \M_\gamma$ satisfies
\[
\codim_\C(\eta(\bfs{m}^\bullet);\M_\gamma) =
	\sum_{i\in[n]} \codim_\C(\Omega_{\bfs {m}^i}(\rowQ{i});\M_{\gamma(i,\bullet)})
\]
where $\M_{\gamma(i,\bullet)}$ denotes the representation space associated to the quiver $\rowQ{i}$ with dimension vector $\gamma(i,\bullet)$. The analogous result holds for vertical strata. \qed
\end{prop}

\begin{defn}
\label{defn:Poincare.series.eta}
For each stratum $\eta = \eta(\bfs{m}^\bullet)$ we define $\cP_\eta$ to be the associated Poincar\'e series $\cP_\eta := \cP[\coho^*_{\G_{\gamma}}(\eta)]$. Similarly, for vertical strata we write $\cP_\theta: = \cP[\coho^*_{\G_\gamma}(\theta)]$.
\end{defn}

\begin{prop}
\label{prop:poincare.eta}
For the Poincar\'e series $\cP_\eta$ we have
\[
\cP_\eta = \prod_{\beta \in \hpr} \cP_{m_\beta}.
\]
The analogous result is true for vertical strata.
\end{prop}

We delay the proof of the above proposition until Section \ref{ss:isotropy.subgps}, but illustrate the results of this section with our running example.

\begin{example}
\label{ex:hstrata.vstrata}
Consider the quiver $S$ with dimension vector $\gamma=\left(\begin{smallmatrix} 2&2\\1&1\end{smallmatrix}\right)$. There are $6$ horizontal strata and $4$ vertical strata. The horizontal strata, their associated Poincar\'e series, and codimensions in $\M_\gamma$ are represented in Table \ref{table:hstrata.2211}. A hexagonal array in the first row of the form
\[
\hstrata{k_{11}}{k_{10}}{k_{01}}{\ell_{10}}{\ell_{01}}{\ell_{11}}
\]
represents the stratum corresponding to the Kostant partition $(k_{10},k_{01},k_{11})$ in the top row and $(\ell_{10},\ell_{01},\ell_{11})$ in the bottom row.
With analogous notation, the vertical strata are tabulated in Table \ref{table:vstrata.2211}.

\begin{table}

{\setlength{\extrarowheight}{1em}
	\begin{tabular}{|l|c|c|c|c|c|c|}
	\hline
	$\eta(\bfs{m}^\bullet)$ &
	$\smallhstrata200001$ & $\smallhstrata111001$ & $\smallhstrata200110$ &
	$\smallhstrata111110$ & $\smallhstrata022001$ & $\smallhstrata022110$ \\[1em]
	\hline
	complex comdimension &
	$0$ & $1$ & $1$ & $2$ & $4$ & $5$ \\[.5em]
	\hline
	Poincar\'e series &
	${\curly P_2 \curly P_1}$ & ${\curly P_1^4}$ & ${\curly P_2 \curly P_1^2}$ &
	${\curly P_1^5}$ & ${\curly P_2^2\curly P_1}$ & ${\curly P_2^2\curly P_1^2}$ \\[.5em]
	\hline
	\end{tabular}
}

\caption{Geometric data corresponding to the six horizontal strata from Example \ref{ex:hstrata.vstrata}.}
\label{table:hstrata.2211}
\end{table}

\begin{table}
{\setlength{\extrarowheight}{1em}
	\begin{tabular}{|l|c|c|c|c|}
	\hline
	$\theta(\bfs{m}^\bullet)$ &
	$\smallvstrata110101$ & $\smallvstrata110210$ &
	$\smallvstrata021101$ & $\smallvstrata021210$ \\[1em]
	\hline
	complex codimension &
	$0$ & $2$ & $2$ & $4$ \\[.5em]
	\hline
	Poincar\'e series &
	$\curly P_{1}^{4}$ & ${\curly P_{2}\curly P_{1}^{3}}$ &
	${\curly P_{2}\curly P_{1}^{3}}$ & ${\curly P_{2}^{2}\curly P_{1}^{2}}$\\[.5em]
	\hline
\end{tabular}
}
\caption{Geometric data corresponding to the four vertical strata from Example \ref{ex:hstrata.vstrata}.}
\label{table:vstrata.2211}
\end{table}

\end{example}

\section{Computation of the superpotential trace function on strata}
\label{s:w.strata}

This section is dedicated to understanding the equivariant geometry of the stratum $\eta\subset\M_\gamma$, with the goal of computing Poincar\'e series in rapid decay cohomology. We begin by recalling a result on Dynkin quivers which we will generalize to our setting.

\begin{lem}\cite[Prop. 3.6]{lfrr2002.duke}
\label{lem:LF.RR.G.hmtpy}
Let $\gamma$ be a dimension vector for a Dynkin quiver and $\Phi$ the positive roots of the underlying root system. For each $\G_\gamma$-orbit $\Omega = \Omega_{\bfs{m}}\subset \M_\gamma$ determined by the Kostant partition $\bfs{m}=(m_\beta)_{\beta\in\Phi}\kp\gamma$ (see Section \ref{ss:dynkin.quivers}), let $\G_\Omega$ denote the stabilizer subgroup $\{g\in \G_\gamma : g\cdot x = x\}$, where $x\in\Omega$ is a generic point. Then, there is a homotopy equivalence of groups $\G_\Omega \hmtpc \prod_{\beta \in \Phi} \U(m_\beta)$.\qed
\end{lem}

\begin{remark}
Observe that $\G_\Omega$ is well-defined up to isomorphism type. In particular, if one chooses to stabilize the point $x'\in\Omega$, there must be some $g\in\G_\gamma$ such that $g\cdot x = x'$. This means that the stabilizers of $x$ and $x'$ are conjugate in $\G_\gamma$, and hence isomorphic.
\end{remark}

We wish to extend Lemma \ref{lem:LF.RR.G.hmtpy} and several of its consequences to the setting of strata for the square product. The first step is the proper replacement for the stabilizer subgroup $\G_\Omega$, and the proper replacement of what should be stabilized.

\subsection{Normal loci and isotropy subgroups}
\label{ss:isotropy.subgps}

Let $\bfs{m}^{\bullet}$ be a horizontal Kostant series for the dimension vector $\gamma$ and let $\eta = \eta(\bfs{m}^{\bullet})$ denote the associated horizontal stratum in $\M_{\gamma}$. Recall from Section \ref{s:stratify.repspace} that along each row, the sequence $\bfs{m}^{i}$ defines a Kostant partition for $\gamma(i,\bullet)$.

\begin{defn}
\label{defn:normal.locus}
The \defin{normal locus of} $\eta=\eta(\bfs{m}^\bullet)$ is the subvariety
\[
	\nu_\eta =
		\{ (x_a)_{a\in Q_1} \in \M_\gamma :
 			\text{
			 for each $i\in[n]$ we have $(x_a)_{a\in\rowQ{i}_1} = \nu_{\bfs{m}^i}$
			 }
		\}.
\]
In words, $\nu_\eta$ is the set of quiver representations which have the normal forms $\nu_{\bfs{m}^i}$ from Section \ref{ss:type.A} along rows in $Q$. It is immediate that $\nu_\eta \subset \eta$. We make the analogous definition for vertical strata.
\end{defn}

\begin{defn}
\label{defn:G.eta}
For each horizontal (or vertical) stratum $\eta$, define the \defin{isotropy subgroup} \[
\G_\eta = \{g \in \G_\gamma : g\cdot\nu_\eta \subset \nu_\eta \}.
\]
\end{defn}

\begin{prop}
\label{prop:G.eta.hmtpy}
Suppose that $\eta$ corresponds to the horizontal Kostant series $\bfs{m}^\bullet = (\bfs{m}^i)$. As in Lemma \ref{lem:LF.RR.G.hmtpy}, let $\G_{i} = \G_{\Omega_{\bfs{m}^i}(\rowQ{i})} \subset \G_{\gamma(i,\bullet)}$. Then up to homotopy equivalence we have
\[
\G_\eta = \prod_{i\in[n]} \G_{i} \hmtpc \prod_{i\in[n]} \prod_{\beta\in \hpri{i}} \U(m^i_\beta) = \prod_{\beta\in\hpr} \U(m_\beta).
\]
For a vertical stratum $\theta$ with vertical Kostant series $\bfs{m}^\bullet$ we have $\G_\theta \hmtpc \prod_{\beta\in\vpr}\U(m_\beta)$.
\end{prop}

\begin{proof}
We will prove the horizontal case; the vertical case is analogous. We can apply Lemma~\ref{lem:LF.RR.G.hmtpy} to each horizontal sub-quiver $\rowQ{i}$ with Kostant partition $\bfs{m}^i$, where in each row we use the stabilizer for the normal form $\nu_{\bfs{m}^i}$ to compute the isomorphism type of $\G_{i}$. Since there is no restriction on the maps along vertical arrows in $\eta$, we obtain a direct product of the resulting groups.
\end{proof}

To proceed, we need the following general lemma about equivariant cohomology.

\begin{lem}
\label{lem:meets.every.orbit}
Let the group $G$ act on the space $X$. Suppose that $A \subset X$ is a subspace with isotropy subgroup $G_A  = \{g\in G:g\cdot A = A\}$. Assume that
\begin{itemize}[leftmargin=*]
\item every $G$-orbit in $X$ intersects $A$;
\item if $g\in G$ is such that there exists $a\in A$ with $g\cdot a \in A$, then $g\in G_A$.
\end{itemize}
Then $\coho^*_G(X) \iso \coho^*_{G_A}(A)$.

Moreover, given a $G$-invariant subspace $Z\subset X$, the result also holds for the cohomology of the pair $(X,Z)$; that is, $\coho^*_G(X,Z) \iso \coho^*_{G_A}(A,Z\intersect A)$.
\end{lem}

\begin{proof}
Let $EG$ be a contractible space with a free $G$ action. The two conditions of the lemma implies that the map $EG \times_{G_A} A \to EG \times_G X$, defined by $[(e,a)]\mapsto [(e,a)]$ is bijective. Then $\coho^*_{G_A}(A) = \coho^*(EG\times_{G_A} A)=\coho^*(EG\times_G X)=\coho^*_G(X)$. The relative result follows similarly.
\end{proof}

\begin{prop}
\label{prop:equivariant.eta.BG.eta}
$\coho^*_{\G_\gamma}(\eta)$ is isomorphic to $\coho^*(B\G_\eta)$.
\end{prop}

\begin{proof}
We have
\[
\coho^*_{\G_\gamma}(\eta)=\coho^*_{\G_\eta}(\nu_\eta) = \coho^*_{\G_\eta}(\hbox{pt})=\coho^*(BG_\eta).
\]
The first equality follows from the observation that for $G=\G_\gamma, X=\eta, A=\nu_\eta$ we have $G_A=\G_\eta$, and that the two conditions of Lemma \ref{lem:meets.every.orbit} are satisfied. The second equality follows because $\nu_\eta$ is a vector space, hence it is (equivariantly) contractible.
\end{proof}

Now, we can establish the claim made in Proposition \ref{prop:poincare.eta}.

\begin{proof}[Proof of Proposition \ref{prop:poincare.eta}]
We have the following sequence of isomorphisms (with tensor products taken over $\R$)
\[
\coho^*_{\G_\gamma}(\eta) \to \coho^*(B\G_\eta) \to \Tensor_{\beta\in\hpr} \coho^*(B\U(m_\beta)).
\]
The first is the result of Proposition \ref{prop:equivariant.eta.BG.eta}. The second follows from using that $B(G\times H) \homeo BG\times BH$ for any groups $G$ and $H$, and then applying Proposition \ref{prop:G.eta.hmtpy} along with the K\"unneth isomorphism. The Poincar\'e series of the last term is $\prod_{\beta\in\hpr} \cP_{m_\beta}$ as desired.
\end{proof}

\subsection{Homotopy of normal loci}
\label{ss:hmtpy.norm.loci}

Recall the cycles $p_{ij}$ defined in Section \ref{ss:not.Q}. Let $\phi_{ij}$ denote the composition $\phi_{p_{ij}}$.

\begin{defn}
For $W$ defined in Section \ref{ss:not.Q}, Equation \eqref{eqn:W.for.Q}, we obtain the regular function
\begin{equation}
\label{eqn:w.gamma.defn}
W_\gamma : \M_{\gamma}\to \C
\text{~given by~}
(\phi_{a})_{a\in Q_{1}} \longmapsto -\sum_{i=1}^{n-1}\sum_{j=1}^{n^\prime-1} \tr(\phi_{ij}).
\end{equation}
We call $W_\gamma$ the \defin{superpotential trace function} of $Q$.
\end{defn}

This is consistent with the choice made in Examples \ref{ex:W.S} and \ref{ex:trace.function.S}. The next step is to compute the superpotential trace function restricted to a stratum. We will focus on horizontal strata, but all the results of this section apply analogously for vertical strata.

Set $N' = n^\prime(n^\prime+1)/2$ and order the positive roots of $A_{n^\prime}$ lexicographically
\begin{equation}
\label{eqn:list.lex.roots}
\beta_1,\beta_2,\ldots,\beta_{N'}
\end{equation}
as in \eqref{eqn:lex.order.roots}. Define $r(a,b):=\floor{\delta(\beta_a,\beta_b)/2}$.

\begin{thm}
\label{thm:our.hmtpy}
Let $\eta$ be a horizontal stratum with horizontal Kostant series $\bfs{m}^\bullet$, and let $P=\nu_\eta \cap W_\gamma^{-1}(S_t)$ for $t<0$ (recall the notation of Section \ref{ss:RDC.defn} regarding rapid decay cohomology). Let
\begin{equation}
\label{eqn:c.eta}
c_\eta = \sum_{i\in [n-1]}\left(\sum_{a,b \in [{N'}]} r(a,b)\,m^{i}_{\beta_a}\,m^{i+1}_{\beta_b}\right)
\end{equation}
There is a $\G_\eta$-equivariant homotopy of pairs
\begin{equation}
\label{eqn:our.hmtpy}
(\nu_\eta,P) \longrightarrow (\R^{2c_\eta},\R^{2c_\eta}-B^{2c_\eta})
\end{equation}
where $B^d$ denotes the $d$-dimensional ball in $\R^d$.
\end{thm}

The goal of the remainder of this subsection is to prove the above theorem. We begin with a lemma regarding linear algebra.

\begin{lem}
\label{lem:herm.form}
Consider the Hermitian form on $\C^p$
\begin{equation}
\mathfrak{h}(x_1,\ldots,x_p) = \sum_{i=1}^{p-1}(x_i^* x_{i+1} + x_{i+1}^* x_i)
\end{equation}
where $z^*$ denotes the complex conjugate of $z$. Set $r=\floor{p/2}$. There is a linear change of coordinates $(x_1,\ldots,x_p) \to (\xi_1,\ldots,\xi_p)$ on $\C^p$ such that $\mathfrak{h}$ takes the form
\[
\sum_{k=1}^{r} |\xi_k|^2 - \sum_{l=r+1}^{2r} |\xi_l|^2.
\]
\end{lem}

\begin{proof}
The form $\mathfrak{h}$ is represented by the $p\times p$ matrix with $1/2$ on the first sub- and super-diagonals, and zeroes everywhere else. Using the ``sum to product'' trigonometric identities, one can verify that the set of eigenvalues for this matrix is $\{\cos(a\pi/(p+1)):a\in [p]\}$. In particular, there are $r$ distinct positive eigenvalues, $r$ distinct negative eigenvalues (in fact the negatives of the positive eigenvalues), and zero occurs as an eigenvalue (with multiplicity one) if and only if $p$ is odd.
\end{proof}

For brevity below we will use the short hand notations $\delta(a,b):=\delta(\beta_a,\beta_b)$, $k(a,b):= k(\beta_a,\beta_b)$, and $\ell(a,b) := \ell(\beta_a,\beta_b)$ (refer to Definition \ref{defn:pos.rts.intersect}).

\begin{proof}[Proof of Theorem \ref{thm:our.hmtpy}]
Because the function $W_\gamma$ effects only adjacent quiver rows, it suffices to prove the result for $n=2$, i.e.~for the quiver $A_2\square A_{n^\prime}$.

Let $\gamma$ denote the dimension vector such that $\bfs{m}^\bullet\kp\gamma$. At each quiver vertex $(i,j)$ let $V_{i,j}$ denote the associated complex $\gamma(i,j)$-dimensional vector space. The formation of the normal locus $\nu_\eta$ amounts to choosing certain bases at each vertex, and hence we may write
\[
\nu_\eta = \prod_{j\in[n^\prime]} \Mat_{\gamma(2,j)\times\gamma(1,j)}(\C)
\]
where
\[
\bfs{X}=(X_1,\ldots,X_{n^\prime}) \in \prod_{j\in[n^\prime]}\Mat_{\gamma(2,j)\times\gamma(1,j)}(\C)
\]
is realized as an element of $\nu_\eta$ by thinking of $X_j$ as the linear map $V_{1,j}\to V_{2,j}$ for a downward pointing arrow $(1,j)\to(2,j)$ and $X_j^\dag$ as the linear map $V_{2,j}\to V_{1,j}$ for an upward pointing arrow $(2,j)\to(1,j)$. Recall that by definition of $\nu_\eta$, the maps along horizontal arrows are determined by the normal forms of Section \ref{ss:type.A}. Without loss of generality, we henceforth assume that vertical arrows connecting $(1,j)$ and $(2,j)$ point downward when $j$ is odd (and thus upward when $j$ is even). A special case of these coordinates is illustrated in Figure \ref{fig:hmtpy.ex}.

\begin{figure}
\begin{tikzpicture}
\node (11) at (0,1.5) {$\C^k$};
\node (12) at (1.5,1.5) {$\C^k$};
\node (13) at (3,1.5) {$\C^k$};
\node (1m-1) at (4.5,1.5) {$\C^k$};
\node (1m) at (6,1.5) {$\C^k$};

\node (21) at (0,0) {$\C^\ell$};
\node (22) at (1.5,0) {$\C^\ell$};
\node (23) at (3,0) {$\C^\ell$};
\node (2m-1) at (4.5,0) {$\C^\ell$};
\node (2m) at (6,0) {$\C^\ell$};

\node at (3.75,.75) {$\cdots$};

\draw[->] (12) -- (11);
\draw[->] (12) -- (13);
\draw[->] (1m-1) -- (1m);
\draw[->] (21) -- (22);
\draw[->] (23) -- (22);
\draw[->] (2m) -- (2m-1);
\draw[->] (11) -- (21);
\draw[->] (22) -- (12);
\draw[->] (13) -- (23);
\draw[->] (2m-1) -- (1m-1);
\draw[->] (1m) -- (2m);

\node at (-.25,.75) {$X_1$};
\node at (1.25,.75) {$X_2^\dag$};
\node at (2.75,.75) {$X_3$};
\node at (5,.75) {$X_{n^\prime-1}^\dag$};
\node at (6.3,.75) {$X_{n^\prime}$};

\node at (0.75,-.25) {id};
\node at (2.25,-.25) {id};
\node at (5.25,-.25) {id};
\node at (0.75,1.75) {id};
\node at (2.25,1.75) {id};
\node at (5.25,1.75) {id};

\end{tikzpicture}
\caption{Coordinates on the normal locus as described in the proof of Theorem \ref{thm:our.hmtpy} for $A_2\square A_{n^\prime}$ and the horizontal stratum corresponding to a multiplicity of $k$ for the longest root in the first row, $\ell$ in the second row, and all other horizontal roots having multiplicity zero.}
\label{fig:hmtpy.ex}
\end{figure}

For each $a\in[{N'}]$ and $\beta_a$ as in Equation \eqref{eqn:list.lex.roots}, let $k(a)$ and $\ell(a)$ be defined by $\beta_a = \beta_{k(a),\ell(a)}$ per the notation of Section \ref{ss:type.A}. The choosing of bases inherent in forming $\nu_\eta$ means that the spaces $V_{i,j}$ decompose as
\begin{equation}
\label{eqn:Vij.decompose}
V_{i,j} = \Dirsum_{a\in[{N'}]} V_{i,j}(a)
\end{equation}
with $V_{i,j}(a)$ consisting of the vectors corresponding to the root $\beta_a$. That is
\[
\dim_\C V_{i,j}(a) = \begin{cases}
m^i_{\beta_a} &\text{if~} j\in[k(a),\ell(a)] \\
0	&\text{otherwise}
\end{cases}.
\]
Set $d(i,j,a) = \dim_\C V_{i,j}(a)$. We can think of each matrix coordinate $X_j$ decomposed into blocks $X_j = \left(X_j(a,b)\right)_{a,b=1}^{N'}$, where if $j$ is odd we have
\[
X_j(a,b) : V_{1,j}(a) \to V_{2,j}(b),
\]
and if $j$ is even we have
\[
X_j(a,b)^\dag : V_{2,j}(b) \to V_{1,j}(a).
\]
Note that $X_j(a,b)$ is a ``zero by zero'' matrix whenever $\beta_a \intersect \beta_b$ has length zero or is empty. By Proposition \ref{prop:G.eta.hmtpy}, the isotropy subgroup $\G_\eta$ is homotopy equivalent to $\prod_{i=1}^2\prod_{a\in[{N'}]} \U(m^i_{\beta_a})$ and we write an element of this group as $(U_{i,a})$ where $i$ runs over the rows of the quiver and $a$ corresponds to the root $\beta_a$; i.e.~$U_{i,a} \in \U(m^i_{\beta_a})$. With these conventions, the action on $\nu_\eta = \{X_j(a,b):j\in[n^\prime],a\in[{N'}],b\in[{N'}]\}$ is
\begin{equation}
\label{eqn:Geta.act.X.coords}
(U_{i,a})_{i\in[2],a\in[{N'}]}\cdot X_j(a,b) = U_{2,b} X_j(a,b) U_{1,a}^\dag.
\end{equation}
Observe that this is compatible with Equation \eqref{eqn:G.acts.M}. We have that
\[
W_\gamma(\bfs{X}) =  -\tr\left(
 X_2^\dag X_1 + X_2^\dag X_3 + X_4^\dag X_3 + \cdots
 \right).
\]
The above sum is finite, but whether the last term is $X^\dag_{n^\prime} X_{n^\prime-1}$ or $X^\dag_{n^\prime-1}X_{n^\prime}$ depends on if $n^\prime$ is even or odd. Let $\til{W} = -\Re[W_\gamma]$ and recall that $P$ is defined to the be set where $\til{W}>-t$. Observe that
\begin{equation}
\label{eqn:tilde.W}
\begin{aligned}
\til{W}(\bfs{X}) &= \frac{1}{2}\sum_{j=1}^{n^\prime-1}
  \tr\left(X_j^\dag X_{j+1} + X_{j+1}^\dag X_j\right) \\
  &= \frac{1}{2} \sum_{a,b\in[{N'}]}\left[ \sum_{j=1}^{n^\prime-1}
  \tr\left(X_j(a,b)^\dag X_{j+1}(a,b) + X_{j+1}(a,b)^\dag X_j(a,b)\right)
  \right].
\end{aligned}
\end{equation}
For each choice of $j,a,b$ such that $d(1,j,a)$ and $d(2,j,b)$ are both nonzero (or equivalently $k(a,b) < \ell(a,b)$ and $X_j(a,b)$ is not an ``empty'' matrix), observe that we have $\ell(a,b) - k(a,b) + 1 = \delta(a,b)$. Since in this case $j\in[k(a),\ell(a)]\intersect[k(b),\ell(b)]$ we actually obtain $d(1,j,a) = m^1_{\beta_a}$ and $d(2,j,b) = m^2_{\beta_b}$ are independent of $j$, and so for the rest of the proof let $d_a$ and $d_b$ denote these numbers respectively. Thus we can write coordinates for the matrices $X_j(a,b)$ as below
\[
X_j(a,b) = \left(x_{e,f}(j,a,b)\right)_{e\in[d_b],f \in [d_a]}
\]
where each $x_{e,f}(j,a,b)$ is an affine coordinate for $\C$. Now, fix $a,b\in[{N'}]$ and consider
\begin{equation}
\label{eqn:fix.ab.trace.x}
\begin{aligned}
&\sum_{j=1}^{n^\prime-1}
	\tr\left(X_j(a,b)^\dag X_{j+1}(a,b) + X_{j+1}(a,b)^\dag X_j(a,b)\right) \\
=&\sum_{v=k(a,b)}^{\ell(a,b)-1}
	\tr\left(X_v(a,b)^\dag X_{v+1}(a,b) + X_{v+1}(a,b)^\dag X_v(a,b)\right) \\
=&\sum_{v}\left[
	\tr\left(X_v(a,b)^\dag X_{v+1}(a,b)\right) + \tr\left(X_{v+1}(a,b)^\dag X_v(a,b)\right)
	\right] \\
=&\sum_{v}\left[
	\sum_{e=1}^{d_b} \sum_{f=1}^{d_a}
	\left( x_{e,f}(v,a,b)^* x_{e,f}(v+1,a,b) + x_{e,f}(v+1,a,b)^* x_{e,f}(v,a,b) \right)
	\right] \\
=& \sum_{e,f} \,
	\mathfrak{h}\left(x_{e,f}(k(a,b),a,b),x_{e,f}(k(a,b)+1,a,b),\ldots,x_{e,f}(\ell(a,b),a,b) \right)
\end{aligned}
\end{equation}
where $\mathfrak{h}$ is the Hermitian form of Lemma \ref{lem:herm.form}. Thus we obtain a linear change of coordinates, which we will denote by $R(a,b)$, of the form \[\left(x_{e,f}(v,a,b)\right)_{v=k(a,b)}^{\ell(a,b)} \stackrel{R(a,b)}{\longrightarrow} \left(\xi_{e,f}(u,a,b)\right)_{u=1}^{\delta(a,b)}\] such that the last line of Equation \eqref{eqn:fix.ab.trace.x} becomes
\begin{equation}
\label{eqn:fix.ab.trace.xi}
\sum_{e,f} \left[
	\sum_{u = 1}^{r(a,b)} |\xi_{e,f}(u,a,b)|^2 - \sum_{u=r(a,b)+1}^{2r(a,b)} |\xi_{e,f}(u,a,b)|^2\right].
\end{equation}
In fact we can define a change coordinates on $\nu_\eta$ by \[\bfs{X}=(X_1,X_2,\ldots,X_{n^\prime}) \longmapsto \bfs{\Xi}=(\Xi_1,\Xi_2,\ldots,\Xi_{n^\prime})\] where the $\gamma(2,j)\times \gamma(1,j)$ complex matrix $\Xi_j$ is a block matrix
\[
\Xi_j = \left( \Xi_j(a,b) \right)_{a,b=1}^{N'}
\]
and the matrices $\Xi_j(a,b)$ are given in coordinates by
\[
\Xi_j(a,b) = \left( \xi_{e,f}(j,a,b) \right)_{e\in[d_b],f\in[d_a]}.
\]
Explicitly, the change of matrix coordinates is also given in terms of matrices by
\[R(a,b)\left( X_{k(a,b)}(a,b),\ldots,X_{\ell(a,b)}(a,b) \right) = \left( \Xi_1(a,b),\ldots,\Xi_{\delta(a,b)}(a,b) \right)\]
for each choice of $a$ and $b$. Since this change is linear, the $\G_\eta$ action on the new coordinates is analogous to Equation \eqref{eqn:Geta.act.X.coords}, i.e.
\begin{equation}
\label{eqn:Geta.act.Xi.coords}
(U_{i,a})_{i\in[2],a\in[{N'}]}\cdot \Xi_j(a,b) = U_{2,b} \Xi_j(a,b) U_{1,a}^\dag.
\end{equation}
Finally, we can define a homotopy $H:\nu_\eta \times [0,1] \to \nu_\eta$ in the new coordinates by
\begin{equation}
\label{eqn:H.defn}
\left( \xi_{e,f}(u,a,b) , s \right) \longmapsto \begin{cases}
\xi_{e,f}(u,a,b), &\text{if~} u\leq r(a,b) \\
(1-s)\xi_{e,f}(u,a,b), &\text{if~}u>r(a,b)
\end{cases}.
\end{equation}
This further implies the analogous formula on the level of matrices
\[
(\Xi_u(a,b) , s) \longmapsto \begin{cases}
\Xi_u(a,b), &\text{if~} u\leq r(a,b) \\
(1-s)\Xi_u(a,b), &\text{if~}u>r(a,b)
\end{cases}.
\]
This together with Equation \eqref{eqn:Geta.act.Xi.coords} implies that $H(g\cdot \bfs{\Xi},s) = g\cdot H(\bfs{\Xi},s)$ for any $g\in \G_\eta$, $\bfs{\Xi}\in\nu_\eta$, and $s\in[0,1]$. That is, the map $H$ is an equivariant homotopy. Observe that $H(-,0)$ is the identity map and that the image of $H(-,1)$ is isomorphic to a complex vector space with coordinates
\[
\{ \xi_{e,f}(u,a,b) \in \C : a,b\in[{N'}], u\leq r(a,b), e\in[d_b], f\in[d_a] \}.
\]
Therefore this image has complex dimension $\sum_{a,b} r(a,b) d_b d_a = c_\eta$. Taking real and imaginary parts of each $\xi_{e,f}(u,a,b)$ above identifies the image of $H(-,1)$ with $\R^{2c_\eta}$.

To see that $H$ is a homotopy \emph{of pairs}, we verify that $H(P,s) \subset P$ for all $s$. Equations \eqref{eqn:fix.ab.trace.xi} and \eqref{eqn:H.defn} imply that
\begin{equation}
\label{eqn:tilW.H.Xi.s}
\til{W} =
\sum_{a,b}\left[
\sum_{u=1}^{r(a,b)}
	\tr\left(\Xi_u(a,b)^\dag \Xi_u(a,b)\right)
-(1-s)^2 \sum_{u=r(a,b)+1}^{2r(a,b)} \tr\left(\Xi_u(a,b)^\dag\Xi_u(a,b) \right)
\right].
\end{equation}
For any $\bfs{\Xi}\in P$, we have that
\[
-t \leq \til{W} = \sum_{a,b}\left[
\sum_{u=1}^{r(a,b)}
	\tr\left(\Xi_u(a,b)^\dag \Xi_u(a,b)\right)
-  \sum_{u=r(a,b)+1}^{2r(a,b)} \tr\left(\Xi_u(a,b)^\dag\Xi_u(a,b) \right)
\right],
\]
and so \eqref{eqn:tilW.H.Xi.s} implies also that $\til{W}|_{H(P,s)}\geq-t$ for all $s$. Hence $H(\bfs\Xi,s)$ is also in $P$. Moreover, when $s=1$ the expression \eqref{eqn:tilW.H.Xi.s} becomes
\[
\begin{aligned}
&\sum_{a,b}\left[\sum_{u\leq r(a,b)}
	\tr\left(\Xi_u(a,b)^\dag \Xi_u(a,b)\right)\right] \\
= &\sum_{a,b}\sum_{u\leq r(a,b)} \sum_{e\in[d_b]} \sum_{f\in[d_a]} |\xi_{e,f}(u,a,b)|^2
= \sum_{i=1}^{c_\eta} |z_i|^2
\end{aligned}
\]
where each $\xi_{e,f}(u,a,b)$ is identified with a complex coordinate $z_i$ above. Passing to real and imaginary parts of the coordinates $z_i$, we see that the image of $P$ in $H(-,1)$ is homeomorphic to $\R^{2c_\eta} - B^{2c_\eta}$ as desired.
\end{proof}

\subsection{Superpotential contributions and shifts}
\label{ss:w.contrib.shifts}

Let $\bfs{m}=(\bfs{m}^1,\ldots,\bfs{m}^n)$ be a horizontal Kostant series and let $\eta$ be the corresponding horizontal stratum. For $\beta' \in \Phi(\rowQ{i'})$ and $\beta'' \in \Phi(\rowQ{i''})$ write
\[
\beta' = \beta^{i'}_{s,t} = \sum_{j=s}^{t} \epsilon_{i'j}
\quad\text{and}\quad
\beta''= \beta^{i''}_{u,v} = \sum_{j=u}^{v} \epsilon_{i''j}
\]
for some $s,t,u,v$.

\begin{defn}
When $i'<i''$, define the numbers \[\upform{\beta'}{\beta''}\quad\text{and}\quad\downform{\beta'}{\beta''}\] respectively to be the number of \emph{upward pointing} and \emph{downward pointing} arrows of the form $(i',j)\to(i'',j)$ where $j$ lies in the overlap $\beta'\intersect\beta'' = [s,t]\intersect[u,v]$, assuming the interval is nonempty. Whenever the overlap is empty, we assign the value zero to both forms.
\end{defn}

Observe that both of these forms vanish whenever $i''-i' \neq 1$. Furthermore, when $\beta'$ and $\beta''$ are thought of as dimension vectors we have \[\lambda(\beta',\beta'') = \downform{\beta'}{\beta''} - \upform{\beta'}{\beta''}.\]

\begin{example}
\label{ex:upform.downform}
In Figure \ref{fig:upform.downform} we see $\downform{\alpha}{\beta}=2$, $\upform{\alpha}{\beta} = \downform{\beta}{\gamma} = \upform{\beta}{\gamma} = 1$, and $\upform{\alpha}{\gamma} = \downform{\alpha}{\gamma} = 0$.

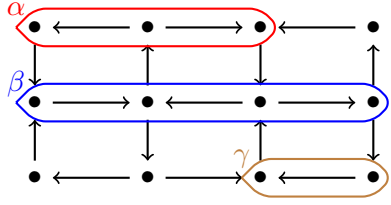
\begin{figure}
\begin{tikzpicture}[thick,auto]

\node (11) at (0,2) {$\bullet$};
\node (21) at (0,1) {$\bullet$};
\node (31) at (0,0) {$\bullet$};
\node (12) at (1.5,2) {$\bullet$};
\node (22) at (1.5,1) {$\bullet$};
\node (32) at (1.5,0) {$\bullet$};
\node (13) at (3,2) {$\bullet$};
\node (23) at (3,1) {$\bullet$};
\node (33) at (3,0) {$\bullet$};
\node (14) at (4.5,2) {$\bullet$};
\node (24) at (4.5,1) {$\bullet$};
\node (34) at (4.5,0) {$\bullet$};

\draw [->] (11) -- (21);
\draw [->] (31) -- (21);
\draw [->] (22) -- (12);
\draw [->] (22) -- (32);
\draw [->] (13) -- (23);
\draw [->] (33) -- (23);
\draw [->] (24) -- (14);
\draw [->] (24) -- (34);
\draw [->] (12) -- (11);
\draw [->] (12) -- (13);
\draw [->] (14) -- (13);
\draw [->] (21) -- (22);
\draw [->] (23) -- (22);
\draw [->] (23) -- (24);
\draw [->] (32) -- (31);
\draw [->] (32) -- (33);
\draw [->] (34) -- (33);

\draw[red,rounded corners]
(-.25,2) -- (0,2.25) -- (3,2.25) -- (3.25,2) -- (3,1.75) -- (0,1.75) -- (-.25,2);
\node[red] at (-.25,2.25) {$\alpha$};

\draw[blue,rounded corners]
(-.25,1) -- (0,1.25) -- (4.5,1.25) -- (4.75,1) -- (4.5,0.75) -- (0,0.75) -- (-.25,1);
\node[blue] at (-.25,1.25) {$\beta$};

\draw[brown,rounded corners]
(2.75,0) -- (3,0.25) -- (4.5,0.25) -- (4.75,0) -- (4.5,-.25) -- (3,-.25) -- (2.75,0);
\node[brown] at (2.75,0.25) {$\gamma$};

\end{tikzpicture}
\caption{Three roots supported along rows in $A_{3}\square A_{4}$.}
\label{fig:upform.downform}
\end{figure}

\end{example}

\begin{defn}
\label{defn:w.contrib}
Let $i'<i''$, $\beta'\in\hpri{i'}$, and $\beta''\in\hpri{i''}$. The \defin{superpotential contribution} $\supc(\beta',\beta'')$ is defined to be the non-negative integer
\begin{equation}
\label{eqn:w.beta.down.up}
\supc(\beta',\beta'') =
\begin{cases}
\downform{\beta'}{\beta''}
	&,~\lambda(\beta',\beta'')\leq 0 \quad\text{(i.e. $\beta'\prec\beta''$)}	\\
\upform{\beta'}{\beta''}
	&,~\text{otherwise}
\end{cases}.
\end{equation}
Notice this can be rephrased as
\begin{equation}
\label{eqn:w.beta.up}
\supc(\beta',\beta'') =
\begin{cases}
\upform{\beta'}{\beta''} + \lambda(\beta',\beta'')
	&,~\lambda(\beta',\beta'')\leq 0 \quad\text{(i.e. $\beta'\prec\beta''$)}	\\
\upform{\beta'}{\beta''}
	&,~\text{otherwise}
\end{cases}.
\end{equation}
\end{defn}

Using this, we define the following combinatorial invariant of the stratum $\eta$, and dedicate the remainder of the subsection to interpreting its geometric significance.

\begin{defn}
The non-negative integer $\w(\eta)$ is called the \defin{superpotential shift} of the horizontal stratum $\eta$. It is defined by the formula
\begin{align}
\label{eqn:w.eta.defn}
\begin{split}
\w(\eta) &= \sum_{1\leq i<j\leq n}
\left(
\sum_{\alpha\in\hpri{i}}\sum_{\beta\in\hpri{j}} \supc(\alpha,\beta)\,m^{i}_\alpha \, m^{j}_\beta
\right)  \\
&= \sum_{i=1}^{n-1} \left( \sum_{\alpha\in\hpri{i}}\sum_{\beta\in\hpri{i+1}} \supc(\alpha,\beta)\,m^{i}_\alpha \, m^{i+1}_\beta \right) .
\end{split}
\end{align}
\end{defn}

Notice that the second equality follows from the first because the superpotential contribution $\supc(\alpha,\beta)$ is nonzero only if $\beta\in\hpri{i+1}$ whenever $\alpha\in\hpri{i}$.

We can similarly define $\w(\theta)$ for vertical strata $\theta=\theta(\bfs{m}^\bullet)$ by counting rightward pointing and leftward pointing arrows. Explicitly, for $j'<j''$ and vertical positive roots $\beta'\in\vprj{j'}$ and $\beta''\in\vprj{j''}$ we analogously define $\rightform{\beta'}{\beta''}$ and $\leftform{\beta'}{\beta''}$. Furthermore, $\lambda(\beta',\beta'') = \rightform{\beta'}{\beta''} - \leftform{\beta'}{\beta''}$, these are nonzero only when $j''-j'=1$. The corresponding superpotential contribution is therefore defined by
\begin{equation}
\label{eqn:w.beta.right.left}
\supc(\beta',\beta'') =
\begin{cases}
\rightform{\beta'}{\beta''}
	&,~\lambda(\beta',\beta'')\leq 0 \quad\text{(i.e. $\beta'\prec\beta''$)}	\\
\leftform{\beta'}{\beta''}
	&,~\text{otherwise}
\end{cases}
\end{equation}
and the superpotential shift is
\[
\w(\theta) = \sum_{j=1}^{m-1} \left( \sum_{\alpha\in\vprj{j}}\sum_{\beta\in\vprj{j+1}} \supc(\alpha,\beta)\,m^{j}_\alpha \, m^{j+1}_\beta \right).
\]

\begin{thm}
\label{thm:poincare.series.stratum}
For each horizontal stratum $\eta$, there is an identity of Poincar\'e series
\begin{equation}
\label{eqn:w.eta.poincare.shift}
\curly P [ \coho^*_{\G_\gamma}(\eta;W_\gamma) ] =
	q^{\w(\eta)}\, \cP_\eta.
\end{equation}
The analogous statement holds for vertical strata.
\end{thm}

\begin{proof}
As for Theorem \ref{thm:our.hmtpy}, it suffices to consider the case $n=2$, i.e.~when $Q=A_2\square A_{n^\prime}$. Moreover, we recycle the notation from the statement and proof of Theorem \ref{thm:our.hmtpy}.

First, (with $t\ll0$) the relative version  of Lemma \ref{lem:meets.every.orbit} provides an isomorphism $\coho^*_{\G_\gamma}(\eta;W_\gamma) \iso \coho^*_{\G_\eta}(\nu_\eta;W_\gamma)$. Next, observe that for each pair $(a,b)$ we have $r(a,b) = \supc(\beta_a,\beta_b)$. Therefore, $\w(\eta) = c_\eta$. This means that, by taking $t\ll0$, the homotopy equivalence of Theorem \ref{thm:our.hmtpy} implies a further isomorphism
\[
\coho^*_{\G_\eta}(\nu_\eta;W_\gamma) \stackrel{\iso}{\longrightarrow} \coho^*_{\G_\eta}(\R^{2\w(\eta)},\R^{2\w(\eta)}-B^{2\w(\eta)})
\]
and the target above is isomorphic to $\coho^*_{\G_\eta}(\R^{2\w(\eta)},\R^{2\w(\eta)}-0)$ via another (equivariant) homotopy equivalence of pairs. Using the Borel construction this latter cohomology is, by definition, $\coho^*(E,E_0)$ where $E$ is a vector bundle over $B\G_\eta$ with fibers $\R^{2\w(\eta)}$ and where $E_0$ denotes the complement of the zero section in the total space of $E$. Now, the Thom Isomorphism Theorem \cite[Theorem~10.4]{jmjs1974} implies that $\coho^j(E,E_0) \iso \coho^{j-2\w(\eta)}(B\G_\eta)$ and hence $\cP[\coho^*_{\G_\gamma}(\eta;W_\gamma)]$ is equal to
\[
\begin{aligned}
\cP[\coho^*_{\G_\eta}(\nu_\eta;W_\gamma)]
	&= \sum_{j\geq 0} q^{j/2} \dim_\R \left(\coho^j_{\G_\eta}(\nu_\eta;W_\gamma)\right)
	= \sum_{k\geq 0} q^k \dim_\R \left(\coho^{2k}_{\G_\eta}(\nu_\eta;W_\gamma)\right) \\
	& = \sum_{k\geq 0} q^{k+\w(\eta)} \dim_\R \left(\coho^{2k}(B\G_\eta)\right)
	= q^{\w(\eta)} \cP[\coho^*(B\G_\eta)]
\end{aligned}
\]
as desired.
\end{proof}

\section{Kazarian spectral sequence in rapid decay cohomology}
\label{s:Kaz.Spec.Seq}

In this section we describe the rapid decay cohomology version of Kazarian's spectral sequence (coming from \cite{mk1997}, see also \cite{rr2013}, c.f.~earlier works e.g.~\cite{marb1983}) for our representation and stratifications. The ordinary cohomology version for Dynkin quiver representations was discussed in \cite{rr2013}.

\begin{thm}
\label{thm:Kaz.spec.seq.RDC}
For each dimension vector $\gamma$ there exists a spectral sequence $E_\bullet^{ij}$ such that
\begin{enumerate}[label=(\alph*),leftmargin=*]
\item \label{item:conv}
$E^{ij}_\bullet$ converges to $\coho^{i+j}_{\G_\gamma}(\M_\gamma;W_\gamma)$;
\item \label{item:hor} taking the direct sum over strata $\eta \in \Hor(\M_\gamma)$ we have
\begin{equation}
\label{eqn:Kaz.spec.seq.hor}
E_1^{ij} = \Dirsum_{\codim_\R(\eta;\M_\gamma)=i} \coho^j_{\G_\gamma}(\eta;W_\gamma) = \Dirsum_{\codim_\R(\eta;\M_\gamma)=i} \coho^{j-2\w(\eta)}(B\G_\eta);
\end{equation}
\item \label{item:E1}
the spectral sequence degenerates at the $E_1$ page.
\end{enumerate}
There is also another spectral sequence, with the same properties except \ref{item:hor} is replaced with
\begin{enumerate}[label=(b'),leftmargin=*]
\item \label{item:ver} taking the direct sum over strata $\theta \in \Ver(\M_\gamma)$ we have
\begin{equation}
\label{eqn:Kaz.spec.seq.ver}
E_1^{ij} = \Dirsum_{\codim_\R(\theta;\M_\gamma)=i} \coho^j_{\G_\gamma}(\theta;W_\gamma) = \Dirsum_{\codim_\R(\theta;\M_\gamma)=i} \coho^{j-2\w(\theta)}(B\G_\theta).
\end{equation}
\end{enumerate}
\end{thm}

\begin{proof}
Consider the group $G:=\G_\gamma$ acting on the vector space $X:=\M_\gamma$ with its stratification to horizontal strata $\eta\in \Hor(\M_\gamma)$ introduced in Section \ref{s:stratify.repspace}. Denoting
\[
X_i = \Union_{\codim_\R(\eta;\M_\gamma)\leq i} \eta
\]
we obtain a topological filtration $\emptyset \subset X_0 \subset X_1 \subset X_2 \subset \cdots \subset X_{\dim_\R X} = X$. Applying the Borel construction ($B_G V = EG\times_G V$ where $EG$ is a contractible space with a free $G$-action) we obtain the filtration
\[
\emptyset \subset B_G X_0 \subset B_G X_1 \subset B_G X_2 \subset \cdots \subset B_G X_{\dim_\R X} = B_G X.
\]
In fact, since $G$ is a product of $GL_n(\C)$'s we may assume that $B_GX$ is a total space of a vector bundle over a product of infinite Grassmannians. That is, $B_GX$ is an infinite dimensional manifold, and the $B_G X_i$'s are finite codimensional open submanifolds.

Let $Z_t=W_\gamma^{-1}(S_t)\subset X$ and recall from Section \ref{ss:rdc.w} that
\[
\coho^*_G(X_i;W_\gamma)=\lim_{t\to -\infty} \coho^*_G(X_i,X_i\cap Z_t)
\]
and the limit is achieved at a small enough $t$. Let us choose $t$ small enough so that $\coho^*_G(X_i,W_\gamma)=\coho^*_G(X_i,X_i\cap Z_t)$ for all of the finitely many $i$.

To simplify notation let us denote $X_i'=B_GX_i$, $Z_i'= B_G( X_i \cap Z_t)$, $X'=B_GX$, $Z'=B_GZ_t$.
The long exact sequence of the triple of quotient spaces
$
X' / Z' \supset
X'_i/Z'_i \supset
X'_{i-1} /Z'_{i-1}
$
induces the exact couple
\begin{equation}\label{eqn:exactcouple}
\xymatrix{
\bigoplus_{i,j} \coho^{i+j}( X'/Z' , X'_i/Z'_i)
\ar[r] & \bigoplus_{i,j} \coho^{i+j}( X'/Z', X'_i/Z'_i) \ar[dl]  \\
 \bigoplus_{i,j} \coho^{i+j}( X'_i/Z'_i, X'_{i-1}/Z'_{i-1}). \ar[u] &
}
\end{equation}
We are going to prove that the spectral sequence associated with this exact couple satisfies the conditions of the theorem.

The spectral sequence of the exact couple converges to the direct limit of the top line of (\ref{eqn:exactcouple}), which is
\[
\coho^*(X'/Z',\emptyset)=\coho^*(B_GX/B_GZ_t)=\coho^*(B_GX,B_GZ_t)=\coho^*_G(X;W_\gamma),
\]
proving property \ref{item:conv}.

The $E_1$ page of the spectral sequence of the exact couple is $E_1^{ij}=\coho^{i+j}(X_i'/Z_i', X_{i-1}'/Z_{i-1}')$, which equals $\coho^{i+j}( X_i', X_{i-1}'\cup Z_i' )$ by an obvious improvement of the standard argument comparing absolute homology of a quotient space with their relative homology (e.g. \cite[Proposition~2.22]{ah2002}). Let $U$ be a tubular neighborhood of $X'_i-X'_{i-1}$ in $X'_i$, and let $U_0=U-(X'_i-X'_{i-1})$. We get
\[
\coho^{i+j}( X_i', X_{i-1}'\cup Z_i' ) =
\coho^{i+j}(U,U_0\cup (Z_i' \cap U))=
\coho^j(X'_i-X'_{i-1},Z'_i \cap (X_i'-X_{i-1}'))
\]
\[
=\bigoplus_{\codim_\R ( \eta;X)=i} \coho^j_G(\eta,\eta\cap W_\gamma^{-1}(S_t))
=\bigoplus_{\codim_\R (\eta;X)=i} \coho^j_G(\eta; W_\gamma).
\]
The first equality holds by excision, the second one by the relative Thom isomorphism \cite[Theorem~11.7.34]{AGP}, and the rest by definition. This completes the proof of the first equality in property \ref{item:hor}. The second equality in property \ref{item:hor} is the content of the proof of Theorem \ref{thm:poincare.series.stratum}.

Property \ref{item:E1} follows from the fact that, according to property \ref{item:hor}, the non-zero terms of the $E_1$-page are all of the form $E_1^{\text{even,even}}$.

The proof of the other spectral sequence, satisfying properties \ref{item:conv}, \ref{item:ver}, and \ref{item:E1} follows the same way starting from the vertical stratification of $\M_\gamma$.
\end{proof}

\subsection{Rapid decay cohomology algebras of general representations; a digression.} 
\label{ss:rdc.general}

The key corollary of Theorem \ref{thm:Kaz.spec.seq.RDC} is that we have an expression for the Poincar\'e series of the rapid decay cohomology algebra $\coho^*_{G_\gamma}(\M_\gamma;W)$ in terms of contributions from the strata. Namely, we have
\[
\cP[ \coho^*_{G_\gamma}(\M_\gamma;W) ] =
\sum_{\eta\in \Hor(\M_\gamma)} q^{\w(\eta)+\codim_{\C}(\eta;\M_\gamma)} \cP_\eta.
\]
In fact we have two such expressions coming from the two spectral sequences of Theorem \ref{thm:Kaz.spec.seq.RDC}, and the comparison of the two is essential in the proof of our main theorem in Section \ref{s:pmt}.

In this subsection, for future reference, we collect the facts that were needed to have such an expression for a rapid decay cohomology algebra coming from a stratification---without refering to the actual representation at hand. The summary of results in Sections \ref{s:stratify.repspace}, \ref{s:w.strata}, \ref{s:Kaz.Spec.Seq} without specifying to the representation $\M_\gamma$ is the following.

\begin{thm}\label{thm:general}
Let $R$ be an algebraic representation of the linear algebraic group $G$, and let $W:R\to \C$ be an invariant regular function (superpotential). Let $R$ be stratified into finitely many $G$-invariant smooth open subvarieties (strata). Let each stratum $\eta$ have a distinguished subset $\nu\subset \eta$ which is a vector space, and denote $G_\eta=\{g\in G: g\nu\subset \nu\}$.
\begin{itemize}
\item  Assume that for all strata $\eta$
\begin{itemize}
\item every $G$ orbit in $\eta$ intersects $\nu$;
\item  if $g\in G$ is such that there exists $a\in \nu$ with $g\cdot a \in \nu$, then $g\in G_\eta$.
\end{itemize}
\item{} Assume that for all $\eta$ the space $BG_\eta$ has only even cohomology.
\item{} Assume that for $t\ll 0$ there is a $G_\eta$-equivariant homotopy equivalence of pairs
\[
(\nu, \nu \cap W^{-1}(S_t)) \simeq (\R^{2\w(\eta)},\R^{2\w(\eta)}-B^{2\w(\eta)}).
\]
\end{itemize}
Then
\[
\pushQED{\qed}
\cP[ \coho^*_G(R;W)]=\sum_\eta q^{\w(\eta)+\codim_{\C}(\eta;R)} \cP[ \coho^*(BG_\eta)].
\qedhere
\popQED
\]
\end{thm}

\section{Counting combinatorial invariants in the quantum algebra}
\label{s:count.qalg}

We begin this section with notation which will appear throughout the sequel. Define the following quadratic forms on the dimension vector $\gamma$:
\begin{align}
\label{eqn:up.gamma.defn}
-\mathbf{up}(\gamma) &= \sum_{c\in Q_1~\mathrm{upward~pointing}} \gamma({t(c)})\gamma({h(c)})
\\
\label{eqn:left.gamma.defn}
-\mathbf{left}(\gamma) &= \sum_{c\in Q_1~\mathrm{leftward~pointing}} \gamma({t(c)})\gamma({h(c)})
\\
\label{eqn:dd.gamma.defn}
-\dd(\gamma) &= \sum_{c\in Q_1~\mathrm{downward~pointing}} \gamma({t(c)})\gamma({h(c)}).
\\
\label{eqn:rr.gamma.defn}
-\rr(\gamma) &= \sum_{c\in Q_1 \text{~rightward pointing}}\gamma(t(c))\gamma(h(c)).
\end{align}
Moreover, we define the \defin{horizontal inner product} and \defin{vertical inner product} respectively to be
\begin{align}
\hip(\gamma) & = -\mathbf{up}(\gamma) - \dd(\gamma) \\
\vip(\gamma) & = -\mathbf{left}(\gamma)-\rr(\gamma).
\end{align}

Furthermore, for any subset of vertices $V\subset Q_0$ for which $\{y_{\epsilon_{ij}}:(i,j)\in V\}$ is a commuting subset in $\A_Q$ and any dimension vector $\gamma$, let
\[
\yy_V^\gamma = \prod_{(i,j)\in V} y_{\epsilon_{ij}}^{\gamma(i,j)}.
\]
Since $V$ forms a commuting set, the product above is well-defined in any order.

\subsection{Counting codimensions in the quantum algebra}
\label{ss:count.codim.qalg}

Fix a row $i$ of $Q$. Order the set of horizontal positive roots $\hpr$ as in \eqref{eqn:hor.ord}. Set ${N'}=n^\prime(n^\prime+1)/2$ and let $\phi_{1}\prec\cdots\prec\phi_{\phi_{N'}}$ be the induced ordering on $\hpri{i}$. Let $\gamma(i,\bullet) = \sum_{j=1}^{n^\prime}\gamma(i,j) \epsilon_{i,j}$ be a dimension vector supported along the row $\rowQ{i}$ and let $\bfs{m}^i = (m_\beta)_{\beta\in\hpri{i}}$ be a Kostant partition of $\gamma(i,\bullet)$ so that $\gamma(i,\bullet) = \sum_{\beta\in\hpri{i}} m_{\beta} \beta$.

For each $\beta\in \hpri{i}$ there are unique non-negative integers $d_\beta^j$ such that
\begin{equation}
\label{eqn:dj.m.gamma}
\beta = \sum_{j=1}^{n^\prime} d_\beta^j \epsilon_{i,j},
\end{equation}
for which we observe that $\gamma(i,j) = \sum_{\beta\in\hpri{i}} m_\beta\,d_\beta^j$. Let $\M_\gamma(i,\bullet)$ and $\Omega_{\bfs{m}^i}(\rowQ{i})$ be as in Proposition \ref{prop:codim.eta}. Further, let $\hhs(i)$ and $\hts(i)$ denote the vertices from row $i$ which happen to be horizontal heads and tails respectively. Note these are, respectively, the sets of sinks and sources in the alternating quiver $Q(i,\bullet)$.

\begin{prop}
\label{prop:rr.qalg.codim}
In addition to the notations outlined above, let $y_j$ denote $y_{\phi_j}$ and let $m_j$ denote $m_{\phi_j}$. The following identity holds in $\hat\A_Q$:
\begin{equation}
\label{eqn:pos.to.simp.codim}
y_{1}^{m_1}\cdots y_{{N'}}^{m_{N'}}
= (-1)^{s_i}\,q^{p_i}\,\yy_{\hhs(i)}^\gamma\,\yy_{\hts(i)}^\gamma
\end{equation}
where
\begin{subequations}
\begin{gather}
\label{eqn:s_i}
s_i = \sum_\beta m_\beta\left(\sum_j d_\beta^j - 1\right),~\text{and} \\
\label{eqn:p_i}
p_i =
\codim_\C\left(\Omega_{\bfs{m}^i}(\rowQ{i});\M_\gamma(i,\bullet)\right)
	+ \frac{1}{2}\sum_{j} \gamma(i,j)^2
	- \frac{1}{2}\sum_{\beta} m_\beta^2
\end{gather}
\end{subequations}
in which the summations are over $j\in[n^\prime]$ and $\beta\in\hpri{i}$.
\end{prop}

\begin{proof}
Since our ordering within rows agrees with \eqref{eqn:rr.pos.roots.order} in the sense of Lemma \ref{lem:rr.order.is.our.order}, this is equivalent to \cite[Lemma~5.1]{rr2013} applied to the subquiver $\rowQ{i}$ with dimension vector $\gamma(i,\bullet)$.
\end{proof}

Observe that $\{y_{\epsilon_{i,j}}:(i,j)\in\hhs(i)\} \subset \hat\A_Q$ forms a commuting subset because the sub-quiver $\rowQ{i}$ is alternating. Similarly, the $\{y_{\epsilon_{i,j}}:(i,j)\in\hts(i)\}$ forms a commuting subset.

\begin{prop}
\label{prop:full.hhs.times.hts}
For any dimension vector $\gamma$ we have
\begin{equation}
\label{eqn:full.hhs.times.hts}
q^{\mathbf{up}(\gamma)} \yy_\hhs^\gamma \,\yy_\hts^\gamma = \yy_{\hhs(1)}^\gamma\,\yy_{\hts(1)}^\gamma\,\yy_{\hhs(2)}^\gamma\,\yy_{\hts(2)}^\gamma\, \cdots \, \yy_{\hhs(n)}^\gamma\,\yy_{\hts(n)}^\gamma.
\end{equation}
\end{prop}

\begin{proof}
Fix $i\in[n-1]$ and let $v\in\hhs(i+1)$. Observe that $y_v$ commutes with $y_u$ for every $u\in\hhs(j)$ with $j\leq i$ and every $u\in\hts(j)$ for $j<i$. Thus, we need to consider the result of commuting $y_v$ past $y_u$ for $u\in \hts(i)$. In fact, $y_u y_v = y_v y_u$ unless $u$ and $v$ are oriented in the quiver as below:
\begin{center}
\begin{tikzpicture}[->,semithick,auto]
\node (1) at (-1.5,.5) {$\bullet$};
\node (2) at (1.5,.5) {$\bullet$};
\node (3) at (-1.5,-.5) {$\bullet$};
\node (4) at (1.5,-.5) {$\bullet$};
\node (u) at (0,.5) {$u$};
\node (v) at (0,-.5) {$v$};
\node (c) at (.2,0) {$c$};

\path
(u) edge (1)
(u) edge (2)
(3) edge (v)
(4) edge (v)
(v) edge (u);
\end{tikzpicture}
\end{center}
In this case, we obtain
\[
y_u\, y_v = q^{\lambda(\epsilon_u,\epsilon_v)}\,y_v\, y_u = q^{-1}\,y_v\, y_u
\]
and hence
\[
y_u^{\gamma(u)}\,y_v^{\gamma(v)} = q^{-\gamma(u)\gamma(v)}\,y_v^{\gamma(v)}\,y_u^{\gamma(u)}.
\]
Observe that $v=t(c)$ and $u=h(c)$ for the upward pointing arrow $c$. Doing this for every $v\in\hhs(i+1)$ and for all $i\in[n-1]$ produces the result.
\end{proof}

The analogous statements to Propositions \ref{prop:rr.qalg.codim} and \ref{prop:full.hhs.times.hts} are true (with analogous proofs) when we consider the columns and vertical strata of the quiver. For example, in that case the analogous statement to Proposition \ref{prop:full.hhs.times.hts} is that
\begin{equation}
\label{eqn:full.hhs.times.hts}
q^{\mathbf{left}(\gamma)} \yy_{\vhs}^\gamma \, \yy_{\vts}^\gamma = \yy_{\vhs(1)}^\gamma\, \yy_{\vts(1)}^\gamma \, \yy_{\vhs(2)}^\gamma\,\yy_{\vts(2)}^\gamma \,\cdots \,\yy_{\vhs(m)}^\gamma\,\yy_{\vts(m)}^\gamma.
\end{equation}

\subsection{Counting superpotential shifts in the quantum algebra}
\label{ss:count.w.qalg}

As in Section \ref{s:mt} let $a = nn^\prime(n^\prime+1)/2$. Let
\begin{equation}
\label{eqn:anam.order}
\beta_{1} \prec \beta_{2} \prec \cdots \prec \beta_{a}
\end{equation}
be an allowed ordering on $\hpr$ according to the rules of Section \ref{s:order.roots}.

\begin{prop}
\label{prop:w.qalg}
Let $\eta$ be a horizontal stratum for $A_{n}\square A_{n^\prime}$ with $\bfs{m}^\bullet$ the corresponding horizontal Kostant series. Let $m_{\beta_l} = m_l$ denote the component of $\bfs{m}^\bullet$ corresponding to the ordered root $\beta_l$ from Equation \eqref{eqn:anam.order}. Let $\gamma = \sum_{l=1}^{a} m_{l}\beta_{l}$ be the resulting dimension vector. The following identity holds in the quantum algebra
\begin{equation}
\label{eqn:w.qalg.hor}
y_{\beta_{1}}^{m_{{1}}}\cdots y_{\beta_a}^{m_{a}}
=
q^{\dd(\gamma)+\w(\eta)}\,
\left(
\prod_{\sigma \in \hpri{1}}^{\rightarrow} y_{\sigma}^{m_{\sigma}}
\right)
\cdots
\left(
\prod_{\tau \in \hpri{n}}^{\rightarrow} y_{\tau}^{m_{\tau}}
\right)
\end{equation}
where the arrows atop each product symbol indicate that the terms are multiplied in the orders induced on $\hpri{1},\ldots, \hpri{n}$ from the ordering of Equation \eqref{eqn:anam.order}.
\end{prop}

\begin{remark}
Observe that $\dd(\gamma)$ can be rewritten as the sum
\begin{subequations}
\label{eqn:dd.gamma.vts}
\begin{align}
-\dd(\gamma)
&=
\sum_{{(i,j)\in \vts}}\gamma(i,j)\gamma(i+1,j)
\\
&=\sum_{{(i+1,j)\in \vhs}}\gamma(i,j)\gamma(i+1,j)
\end{align}
\end{subequations}
where we interpret $\gamma(l,k) = 0$ if $l>n$ or $k>n^\prime$. Furthermore, the analogous result holds for a vertical stratum and the $\beta_l$ vertical positive roots. There are $b = n^\prime n(n+1)/2$ such roots. In this case
\begin{equation}
\label{eqn:w.qalg.vert}
y_{\beta_{1}}^{m_{{1}}}\cdots y_{\beta_{b}}^{m_{{b}}}
=
q^{\rr(\gamma)+\w(\eta)}\,
\left(
\prod_{\sigma \in \vprj{1}}^{\rightarrow} y_{\sigma}^{m_{\sigma}}
\right)
\cdots
\left(
\prod_{\tau \in \vprj{m}}^{\rightarrow} y_{\tau}^{m_{\tau}}
\right)
\end{equation}
where the expression for $\rr(\gamma)$ can be rewritten as
\begin{subequations}
\label{eqn:rr.gamma.hts}
\begin{align}
-\rr(\gamma)
&=\sum_{{(i,j)\in \hts}}\gamma(i,j)\gamma(i,j+1)
\\
&=\sum_{{(i,j+1)\in \hhs}}\gamma(i,j)\gamma(i,j+1).
\end{align}
\end{subequations}
The proof of Equation \eqref{eqn:w.qalg.vert} is completely analogous to Proposition \ref{prop:w.qalg}, which we now give.
\end{remark}

\begin{proof}[Proof of Proposition \ref{prop:w.qalg}]
Define $p$ by the equation
\begin{equation}
\label{eqn:p.define}
y_{\beta_{1}}^{m_{{1}}}\cdots y_{\beta_a}^{m_{a}}
=
q^{p}\,
\left(
\prod_{\sigma \in \hpri{1}}^{\rightarrow} y_{\sigma}^{m_{\sigma}}
\right)
\cdots
\left(
\prod_{\tau \in \hpri{n}}^{\rightarrow} y_{\tau}^{m_{\tau}}
\right).
\end{equation}
Notice that if $\beta' \in \hpri{i'}$ and $\beta'' \in \hpri{i''}$ with $|i'-i''|\geq 2$, then $y_{\beta'}$ and $y_{\beta''}$ commute in $\hat{\A}_\Q$. Hence the only contributions to $p$ come from commuting $y_{\beta'}$ past $y_{\beta''}$ on the left-hand side when $i''-i'=1$. Thus we may reduce to the case $n=2$, i.e.\ to the quiver $A_2\square A_{n^\prime}$.

Throughout the rest of the proof, we let $\sigma\in\hpri{1}$ and $\tau\in\hpri{2}$. Observe that the situation
\[
\cdots y_{\tau}^{m_\tau} \cdots y_{\sigma}^{m_\sigma} \cdots
\]
occurs on the left-hand side of \eqref{eqn:p.define} only if $\lambda(\tau,\sigma)\leq 0$ due to our choice of ordering in \eqref{eqn:hor.ord} and consequently in \eqref{eqn:anam.order}. Motivated by this, set
\begin{gather*}
\Neg = \{(\sigma,\tau):\sigma\in\hpri{1},\tau\in\hpri{2},\lambda(\tau,\sigma)\leq 0\},
\\
\Pos = \{(\sigma,\tau):\sigma\in\hpri{1},\tau\in\hpri{2},\lambda(\tau,\sigma)>0\}
\end{gather*}
so that we see $\hpri{1}\times\hpri{2}$ is the disjoint union $\Neg\union\Pos$. Moreover, we can now write
\[
p = \sum_{(\sigma,\tau)\in \Neg} \lambda(\tau,\sigma) m_\sigma m_\tau.
\]
On the other hand, we have by Equation \eqref{eqn:w.beta.down.up}
\begin{align*}
\w(\eta) & = \sum_{\sigma\in\hpri{1}} \sum_{\tau\in\hpri{2}} \supc(\sigma,\tau)m_\sigma m_\tau \\
	&= \sum_{(\sigma,\tau)\in \Neg} \supc(\sigma,\tau)m_\sigma m_\tau
		+ \sum_{(\sigma,\tau)\in \Pos} \supc(\sigma,\tau)m_\sigma m_\tau \\
	&= \sum_{(\sigma,\tau)\in \Neg} \upform{\sigma}{\tau} m_\sigma m_\tau
		+ \sum_{(\sigma,\tau)\in \Pos} \downform{\sigma}{\tau} m_\sigma m_\tau.
\end{align*}
Furthermore, we have
\begin{align*}
-\dd(\gamma) & = \sum_{(1,j)\in\vts}\gamma(1,j)\gamma(2,j)
	= \sum_{(1,j)\in\vts}\left(\sum_{j\in\sigma} m_\sigma \right)\left( \sum_{j\in\tau} m_\tau\right)\\
	\intertext{where by $j\in \beta$ for a root $\beta$, we mean that $j$ is in the interval $[k(\beta),\ell(\beta)]$; see Definition \ref{defn:pos.rts.intersect}. This is further equal to}
	&= \sum_{(1,j)\in\vts} \left(\sum_{j\in\sigma,\tau} m_\sigma m_\tau \right)
	= \sum_{(\sigma,\tau)\in\hpri{1}\times\hpri{2}} \downform{\sigma}{\tau}m_\sigma m_\tau
\end{align*}
where in the last equality we have changed the order of summation. Putting together our expressions for $\w(\eta)$ and $\dd(\gamma)$, we obtain that $\dd(\gamma)+\w(\eta)$ is equal to
\begin{align*}
	& -\left( \sum_{\sigma\in\hpri{1}}\sum_{\sigma\in\hpri{2}} \downform{\sigma}{\tau}m_\sigma m_\tau \right)
	 	+ \left( \sum_{(\sigma,\tau)\in \Neg} \upform{\sigma}{\tau} m_\sigma m_\tau \right)
		+ \left( \sum_{(\sigma,\tau)\in \Pos} \downform{\sigma}{\tau} m_\sigma m_\tau \right)  \\
	=& \left(\sum_{(\sigma,\tau)\in \Neg} \upform{\sigma}{\tau} m_\sigma m_\tau\right)
		- \left(\sum_{(\sigma,\tau)\in \Neg} \downform{\sigma}{\tau} m_\sigma m_\tau\right)
	= \sum_{(\sigma,\tau)\in \Neg}
		\left(\upform{\sigma}{\tau} - \downform{\sigma}{\tau} \right) m_\sigma m_\tau
\end{align*}
which is further equal to $ \sum_{(\sigma,\tau)\in \Neg} \lambda(\tau,\sigma) m_\sigma m_\tau = p,$ as desired.
\end{proof}

\section{Proof of the main theorem}
\label{s:pmt}

Fix a dimension vector $\gamma$. Theorem \ref{thm:Kaz.spec.seq.RDC} implies that the Poincar\'e series $\cP[\coho^*_{\G_\gamma}(\M_\gamma;W_\gamma)]$ is equal to the sum over horizontal strata
\begin{subequations}
\begin{align}
&\sum_{\eta\in\Hor(\M_\gamma)} q^{\w(\eta)+\codim_\C(\eta;\M_\gamma)} \cP_\eta
\label{eqn:hor.q.series} \\
\intertext{and the sum over vertical strata}
&\sum_{\theta\in\Ver(\M_\gamma)} q^{\w(\theta)+\codim_\C(\theta;\M_\gamma)} \cP_\theta.
\label{eqn:ver.q.series}
\end{align}
\end{subequations}
That is, we have the following corollary to Theorem \ref{thm:Kaz.spec.seq.RDC}.

\begin{cor}
\label{cor:poinc.h.equals.poinc.v}
For every dimension vector $\gamma$, we have the $q$-series identity
\[
\pushQED{\qed}
 \sum_{\eta\in\Hor(\M_\gamma)} q^{\w(\eta)+\codim_\C(\eta;\M_\gamma)} \cP_\eta
 	=
 \sum_{\theta\in\Ver(\M_\gamma)} q^{\w(\theta)+\codim_\C(\theta;\M_\gamma)} \cP_\theta.
\qedhere\popQED
\]
\end{cor}

Let $y_{ij}$ denote the element $y_{\epsilon_{ij}}$ in the quantum algebra and set $\bfs{y} = \prod_{i,j} y_{ij}^{\gamma(i,j)}$ where the product is taken over a fixed ordering of the variables $y_{ij}$.

The remainder of this section will be dedicated to proving that \eqref{eqn:hor.q.series} is the coefficient of $\bfs{y}$ on the left-hand side of Equation \eqref{eqn:mt} up to a power of $q$ which depends only on $\gamma$ and the fixed ordering chosen in the product $\bfs{y}$. Similarly \eqref{eqn:ver.q.series} is the coefficient of $\bfs{y}$ on the right-hand side of \eqref{eqn:mt} up to the \emph{same} power of $q$. Because of Corollary \ref{cor:poinc.h.equals.poinc.v}, this will prove our main Theorem \ref{thm:mt}. First, we verify Corollary \ref{cor:poinc.h.equals.poinc.v} in the case of our running example.

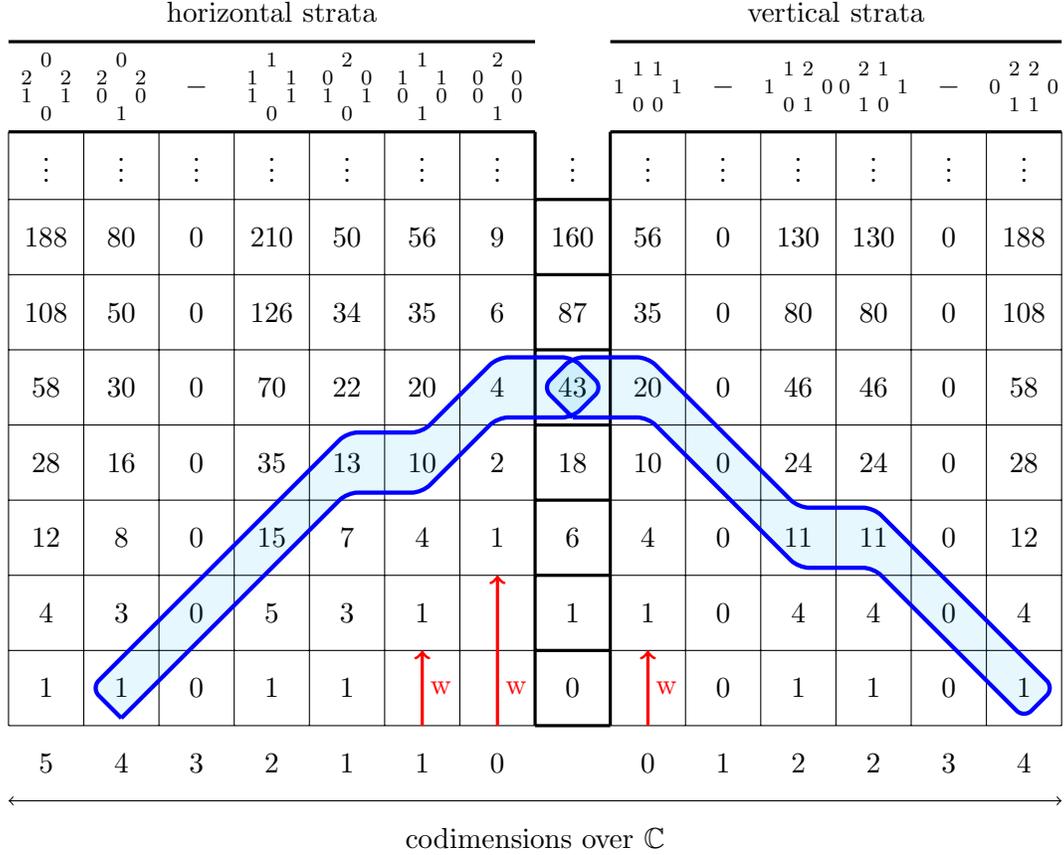
\begin{figure}
	\begin{tikzpicture}
		\draw[step=1cm,black,thin] (0,0) grid (7,7.9);
		\draw[step=1cm,black,thin] (8,0) grid (14,7.9);
		\draw[step=1cm,black,very thick] (7,0) grid (8,7.9);
		
		\draw[<->] (0,-1)--(14,-1);
		
		\draw[->,color=red, very thick](6.5,0)--(6.5,2);
		\node[color=red] at (6.75,0.5) {$\mathrm{w}$};
		
		\draw[->,color=red, very thick](5.5,0)--(5.5,1);
		\node[color=red] at (5.75,0.5) {$\mathrm{w}$};
		
		\draw[->,color=red, very thick](8.5,0)--(8.5,1);
		\node[color=red] at (8.75,0.5) {$\mathrm{w}$};
		
		\node at (3.5,9.5) {horizontal strata};
		\draw[very thick](0,9.1)--(7,9.1);
		\draw[very thick](0,7.9)--(7,7.9);
		\node at (6.5,8.5) {$\smallhstrata200001$};
		\node at (5.5,8.5) {$\smallhstrata111001$};
		\node at (4.5,8.5) {$\smallhstrata200110$};
		\node at (3.5,8.5) {$\smallhstrata111110$};
		\node at (2.5,8.5) {$-$};
		\node at (1.5,8.5) {$\smallhstrata022001$};
		\node at (0.5,8.5) {$\smallhstrata022110$};
		
		\node at (11,9.5) {vertical strata};
		\draw[very thick](8,9.1)--(14,9.1);
		\draw[very thick](8,7.9)--(14,7.9);
		\node at (8.5,8.5) {$\smallvstrata110101$};
		\node at (9.5,8.5) {$-$};
		\node at (10.5,8.5) {$\smallvstrata110210$};
		\node at (11.5,8.5) {$\smallvstrata021101$};
		\node at (12.5,8.5) {$-$};
		\node at (13.5,8.5) {$\smallvstrata021210$};
		
		\node at (0.5,7.5) {$\vdots$};
		\node at (1.5,7.5) {$\vdots$};
		\node at (2.5,7.5) {$\vdots$};
		\node at (3.5,7.5) {$\vdots$};
		\node at (4.5,7.5) {$\vdots$};
		\node at (5.5,7.5) {$\vdots$};
		\node at (6.5,7.5) {$\vdots$};
		\node at (7.5,7.5) {$\vdots$};
		\node at (8.5,7.5) {$\vdots$};
		\node at (9.5,7.5) {$\vdots$};
		\node at (10.5,7.5) {$\vdots$};
		\node at (11.5,7.5) {$\vdots$};
		\node at (12.5,7.5) {$\vdots$};
		\node at (13.5,7.5) {$\vdots$};

		
		\node at (6.5,2.5) {$1$};
		\node at (6.5,3.5) {$2$};
		\node at (6.5,4.5) {$4$};
		\node at (6.5,5.5) {$6$};
		\node at (6.5,6.5) {$9$};
		
		\node at (5.5,1.5) {$1$};
		\node at (5.5,2.5) {$4$};
		\node at (5.5,3.5) {$10$};
		\node at (5.5,4.5) {$20$};
		\node at (5.5,5.5) {$35$};
		\node at (5.5,6.5) {$56$};
		
		\node at (4.5,0.5) {$1$};
		\node at (4.5,1.5) {$3$};
		\node at (4.5,2.5) {$7$};
		\node at (4.5,3.5) {$13$};
		\node at (4.5,4.5) {$22$};
		\node at (4.5,5.5) {$34$};
		\node at (4.5,6.5) {$50$};
		
		\node at (3.5,0.5) {$1$};
		\node at (3.5,1.5) {$5$};
		\node at (3.5,2.5) {$15$};
		\node at (3.5,3.5) {$35$};
		\node at (3.5,4.5) {$70$};
		\node at (3.5,5.5) {$126$};
		\node at (3.5,6.5) {$210$};
		
		\node at (2.5,0.5) {$0$};
		\node at (2.5,1.5) {$0$};
		\node at (2.5,2.5) {$0$};
		\node at (2.5,3.5) {$0$};
		\node at (2.5,4.5) {$0$};
		\node at (2.5,5.5) {$0$};
		\node at (2.5,6.5) {$0$};
		
		\node at (1.5,0.5) {$1$};
		\node at (1.5,1.5) {$3$};
		\node at (1.5,2.5) {$8$};
		\node at (1.5,3.5) {$16$};
		\node at (1.5,4.5) {$30$};
		\node at (1.5,5.5) {$50$};
		\node at (1.5,6.5) {$80$};
		
		\node at (0.5,0.5) {$1$};
		\node at (0.5,1.5) {$4$};
		\node at (0.5,2.5) {$12$};
		\node at (0.5,3.5) {$28$};
		\node at (0.5,4.5) {$58$};
		\node at (0.5,5.5) {$108$};
		\node at (0.5,6.5) {$188$};
		
		\node at (8.5,1.5) {$1$};
		\node at (8.5,2.5) {$4$};
		\node at (8.5,3.5) {$10$};
		\node at (8.5,4.5) {$20$};
		\node at (8.5,5.5) {$35$};
		\node at (8.5,6.5) {$56$};

		\node at (9.5,0.5) {$0$};
		\node at (9.5,1.5) {$0$};
		\node at (9.5,2.5) {$0$};
		\node at (9.5,3.5) {$0$};
		\node at (9.5,4.5) {$0$};
		\node at (9.5,5.5) {$0$};
		\node at (9.5,6.5) {$0$};
		
		\node at (10.5,0.5) {$1$};
		\node at (10.5,1.5) {$4$};
		\node at (10.5,2.5) {$11$};
		\node at (10.5,3.5) {$24$};
		\node at (10.5,4.5) {$46$};
		\node at (10.5,5.5) {$80$};
		\node at (10.5,6.5) {$130$};
		
		\node at (11.5,0.5) {$1$};
		\node at (11.5,1.5) {$4$};
		\node at (11.5,2.5) {$11$};
		\node at (11.5,3.5) {$24$};
		\node at (11.5,4.5) {$46$};
		\node at (11.5,5.5) {$80$};
		\node at (11.5,6.5) {$130$};
		
		\node at (12.5,0.5) {$0$};
		\node at (12.5,1.5) {$0$};
		\node at (12.5,2.5) {$0$};
		\node at (12.5,3.5) {$0$};
		\node at (12.5,4.5) {$0$};
		\node at (12.5,5.5) {$0$};
		\node at (12.5,6.5) {$0$};
		
		\node at (13.5,0.5) {$1$};
		\node at (13.5,1.5) {$4$};
		\node at (13.5,2.5) {$12$};
		\node at (13.5,3.5) {$28$};
		\node at (13.5,4.5) {$58$};
		\node at (13.5,5.5) {$108$};
		\node at (13.5,6.5) {$188$};

		
		\node at (7,-1.5) {codimensions over $\C$};
		\node at (0.5,-0.5) {$5$};
		\node at (1.5,-0.5) {$4$};
		\node at (2.5,-0.5) {$3$};
		\node at (3.5,-0.5) {$2$};
		\node at (4.5,-0.5) {$1$};
		\node at (5.5,-0.5) {$1$};
		\node at (6.5,-0.5) {$0$};
		
		\node at (8.5,-0.5) {$0$};
		\node at (9.5,-0.5) {$1$};
		\node at (10.5,-0.5) {$2$};
		\node at (11.5,-0.5) {$2$};
		\node at (12.5,-0.5) {$3$};
		\node at (13.5,-0.5) {$4$};
		
		\node at (7.5,0.5) {$0$};
		\node at (7.5,1.5) {$1$};
		\node at (7.5,2.5) {$6$};
		\node at (7.5,3.5) {$18$};
		\node at (7.5,4.5) {$43$};
		\node at (7.5,5.5) {$87$};
		\node at (7.5,6.5) {$160$};
		
		\filldraw[fill=cyan, fill opacity=0.1, draw=blue, ultra thick, rounded corners]
		(1.5,.1) -- (1.1,.5) -- (1.5,.9) --
		(4.5,3.9) -- (5.5,3.9) -- (6.5, 4.9) --
		(7.5,4.9) -- (7.9,4.5) -- (7.5,4.1) --
		(6.5,4.1) -- (5.5,3.1) -- (4.5,3.1) -- (1.5,.1);
		
		\filldraw[fill=cyan, fill opacity=0.1, draw=blue, ultra thick, rounded corners]
		(7.5,4.1) -- (7.1,4.5) -- (7.5,4.9) --
		(8.5,4.9) -- (10.5,2.9) -- (11.5,2.9) --
		(13.5,.9) -- (13.9,.5) -- (13.5,.1) --
		(11.5,2.1) -- (10.5,2.1) -- (8.5,4.1) -- (7.5,4.1);
			
		\end{tikzpicture}
\caption{Betti tables illustrating the Kazarian spectral sequence argument applied to the horizontal and vertical strata for $A_2\square A_2$ as in Example \ref{ex:qseries.ident}. The number in the $r$-th row (counting from the bottom) and in the column corresponding to the stratum $\eta$ is the coefficient of $q^r$ in the series expansion of $q^{\w(\eta)}\,\cP_\eta$. Summing either the numbers coming from horizontal or vertical strata inside the blue ``snake'' results in the answer $43$.}
\label{fig:betti.tables}
\end{figure}

\begin{example}
\label{ex:qseries.ident}
Again consider $S$ with the dimension vector $\gamma=\left(\begin{smallmatrix} 2&2\\1&1\end{smallmatrix}\right)$ as in Example \ref{ex:hstrata.vstrata}. We obtain the following $q$-series identity
\begin{align}
\label{eqn:2211.qseries}
\begin{split}
q^{2}\,\curly P_2 \curly P_1 + q^2\,\curly P_1^4 + q^1\,\curly P_2 \curly P_1^2 &+ q^2\,\curly P_1^5 + q^4\,\curly P_2^2\curly P_1 + q^5\,\curly P_2^2\curly P_1^2 \\
&=
q^1\curly P_{1}^{4} + q^2\,\curly P_{2}\curly P_{1}^{3} + q^2\,\curly P_{2}\curly P_{1}^{3} + q^4\curly P_{2}^{2}\curly P_{1}^{2}
\end{split}
\end{align}
where each term $q^{p}\,\curly P_{2}^{l}\curly P_{1}^{k}$ corresponds to data on a single stratum. The exponent $p$ is the degree shift required by the sum of the codimension and superpotential shift of that stratum. For each horizontal (respectively vertical) stratum $\eta$ (resp.~$\theta$), the product $\cP_2^l\cP_1^k$ is the Poincar\'e series $\cP_\eta$ (resp.~$\cP_\theta$). The terms on the left correspond to the $6$ horizontal strata and the terms on the right correspond to the $4$ vertical strata. This is illustrated in the Betti table of Figure \ref{fig:betti.tables}. Observe that without the superpotential contributions $\w$, summing diagonally in the Betti table still gives a true identity (this would be the Betti table for ordinary cohomology, not rapid decay), which corresponds to Equation \eqref{eqn:55term}, i.e.~it is \emph{not} independent of the pentagon identity. The superpotential contributions are the added ingredient for the Keller identity, c.f.~Equation \eqref{eqn:55termKeller}. Alternatively, the series identity \eqref{eqn:2211.qseries} is also obtained by computing the coefficient of $y_{1}^{2}y_{2}^{2}y_{3}^{1}y_{4}^{1}$ on both sides of Equation \eqref{eqn:dilog.S.2211}.
\end{example}

\begin{lem}
\label{lem:switch.HH.HT}
We have the following identity in $\A_Q$.
\begin{equation}
\label{eqn:switch.HH.HT}
  \yy_\vhs^\gamma \, \yy_\vts^\gamma  =
 	q^{\vip(\gamma) - \hip(\gamma)}\,
 		\yy_\hhs^\gamma \, \yy_\hts^\gamma.
\end{equation}
\end{lem}

\begin{proof}
We have that $\vhs = \hts$ and $\hhs = \vts$, so we need to commute the two $\yy$ products. Throughout the proof, let $y_{i,j} = y_{\epsilon_{i,j}}$. Consider a vertex $(i,j)$ which is a horizontal head. It appears in the quiver as below:
\begin{center}
\begin{tikzpicture}[->,semithick,auto]
\node (ij) at (0,0) {$(i,j)$};
\node (ij-1) at (-1.7,0) {$(i,j-1)$};
\node (i-1j) at (0,1) {$(i-1,j)$};
\node (i+1j) at (0,-1) {$(i+1,j)$};
\node (ij+1) at (1.7,0) {$(i,j+1)$};

\path
(ij) edge (i-1j)
(ij) edge (i+1j)
(ij-1) edge (ij)
(ij+1) edge (ij);
\end{tikzpicture}
\end{center}
Each vertex above which is not $(i,j)$ lies in the set $\vhs=\hts$, and these are the only vertices $(k,l)$ such that $y_{i,j}$ and $y_{k,l}$ do not commute. We need to compute the number $p_{ij}$ obtained by the following commutation operation
\[
y_{i,j-1}^{\gamma(i,j-1)} y_{i,j+1}^{\gamma(i,j+1)} y_{i-1,j}^{\gamma(i-1,j)} y_{i+1,j}^{\gamma(i+1,j)} y_{i,j}^{\gamma(i,j)} = q^{p_{ij}}\, y_{i,j}^{\gamma(i,j)} y_{i,j-1}^{\gamma(i,j-1)} y_{i,j+1}^{\gamma(i,j+1)} y_{i-1,j}^{\gamma(i-1,j)} y_{i+1,j}^{\gamma(i+1,j)}
\]
which is given by
\[
\begin{aligned}
p_{ij} = & \lambda(\epsilon_{i,j-1},\epsilon_{i,j})\gamma(i,j-1)\gamma(i,j) +
			\lambda(\epsilon_{i,j+1},\epsilon_{i,j})\gamma(i,j+1)\gamma(i,j) \\
	 	& + \lambda(\epsilon_{i-1,j},\epsilon_{i,j})\gamma(i-1,j)\gamma(i,j) +
			\lambda(\epsilon_{i+1,j},\epsilon_{i,j})\gamma(i+1,j)\gamma(i,j) \\
	=	& \gamma(i,j-1)\gamma(i,j) + \gamma(i,j+1)\gamma(i,j)
		 - \gamma(i-1,j)\gamma(i,j) - \gamma(i+1,j)\gamma(i,j).
\end{aligned}
\]
Observe that the result above is the sum of $\gamma(t(c))\gamma(h(c))$ for each horizontal arrow $c$ incident with $(i,j)$ \emph{minus} the sum of $\gamma(t(c))\gamma(h(c))$ for each vertical arrow $c$ incident with $(i,j)$. Performing this computation for each $(i,j)\in\hhs$ gives the result since $\vip(\gamma)$ and $\hip(\gamma)$ are respectively the sums of $\gamma(t(c))\gamma(h(c))$ over \emph{all} horizontal and vertical arrows.
\end{proof}

We are now ready to prove the main theorem, that is, to verify the equality \eqref{eqn:mt}.

\begin{proof}[Proof of Theorem \ref{thm:mt}]
We will work with the left-hand side of Equation \eqref{eqn:mt} which corresponds to the horizontal strata; the computations on the right-hand side are analogous, but done with vertical strata.

Fix a dimension vector $\gamma$ and consider the coefficient of $y_\gamma$ on the left-hand side of \eqref{eqn:mt}. Observe that by Equation \eqref{eqn:E.defn.poincare} this is given by
\begin{equation}
\label{eqn:orig.LHS}
\sum_{\bfs{m}^\bullet\kp \gamma} \left[ (-1)^{\sum_{\phi\in\hpr}m_\phi}\, q^{\frac{1}{2}\sum_{\phi\in\hpr}m_\phi^2}\,\left( \prod_{\phi\in\hpr}\cP_{m_\phi} \right) \left( y_{\phi_1}^{m_{\phi_1}} y_{\phi_2}^{m_{\phi_2}} \cdots y_{\phi_a}^{m_{\phi_a}} \right) \right]
\end{equation}
where the sum is over all horizontal Kostant series $\bfs{m}^\bullet$ of $\gamma$, and $\phi_1\prec\phi_2\prec \cdots \prec\phi_a$ is the ordering from \eqref{eqn:hor.ord}. Now fix a particular Kostant series $\bfs{m}^\bullet$, write $\eta = \eta(\bfs{m}^\bullet)$ for the corresponding horizontal stratum, and analyze the term
\[
(-1)^{\sum_{\phi\in\hpr}m_\phi}\, q^{\frac{1}{2}\sum_{\phi\in\hpr}m_\phi^2}\,\left( \prod_{\phi\in\hpr}\cP_{m_\phi} \right) \left( y_{\phi_1}^{m_{\phi_1}} y_{\phi_2}^{m_{\phi_2}} \cdots y_{\phi_a}^{m_{\phi_a}} \right)
\]
{which, by Definition \ref{defn:Poincare.series.eta}, is further equal to}
\[
 (-1)^{\sum_{\phi\in\hpr}m_\phi}\, q^{\frac{1}{2}\sum_{\phi\in\hpr}m_\phi^2}\,\cP_\eta\, \left( y_{\phi_1}^{m_{\phi_1}} y_{\phi_2}^{m_{\phi_2}} \cdots y_{\phi_a}^{m_{\phi_a}} \right)
\]
{and using the result of Proposition \ref{prop:w.qalg} becomes}
\[
 (-1)^{\sum_{\phi\in\hpr}m_\phi}\, q^{\frac{1}{2}\sum_{\phi\in\hpr}m_\phi^2 + \dd(\gamma)+\w(\eta)}\,\cP_\eta\,\left(\left[\prod_{\sigma\in\hpri{1}} y_\sigma^{m_\sigma}\right] \cdots \left[\prod_{\sigma\in\hpri{n}} y_\sigma^{m_\sigma}\right] \right).
\]
Next, we apply Proposition \ref{prop:rr.qalg.codim} to $\prod_{\sigma\in\hpri{i}}y_\sigma^{m_\sigma}$ for each $i$ to obtain a product of the form
\begin{equation}
\label{eqn:s.prime.p.prime}
(-1)^{s} \, q^{p} \, \cP_\eta \, \left(\yy_{\hhs(1)}^\gamma\,\yy_{\hts(1)}^\gamma\right) \cdots \left(\yy_{\hhs(n)}^\gamma \, \yy_{\hts(n)}^\gamma\right).
\end{equation}
We recall Equation \eqref{eqn:s_i} to see that $s$ is given by
\begin{align*}
s &= \sum_{\phi\in\hpr}m_\phi + \sum_{i\in[n]} s_i \\
	& = \sum_{\phi\in\hpr}m_\phi + \sum_{i\in[n]} \left(\sum_{\phi\in\hpri{i}} m_\phi \left(\sum_{j\in[n^\prime]} d^j_{\phi} - 1\right) \right) \\
	& = \sum_{i\in[n]}\sum_{\phi\in\hpri{i}} m_\phi\,d^j_\phi = \sum_{i\in[n],j\in[n^\prime]}\gamma(i,j),
\end{align*}
where the numbers $d^j_\phi$ are defined as at the beginning of Section \ref{ss:count.codim.qalg} and the last equality comes from Equation \eqref{eqn:dj.m.gamma}. Hence $s$ depends \emph{only} on $\gamma$, and not on the Kostant series $\bfs{m}^\bullet$. The importance of this fact is that we can factor the sign $(-1)^s$ out of the summation over $\bfs{m}^\bullet \kp \gamma$ in Equation \eqref{eqn:orig.LHS}. To determine the value of $p$ in Equation \eqref{eqn:s.prime.p.prime} we refer to Equation \eqref{eqn:p_i} to obtain
\begin{align*}
p & = \frac{1}{2}\sum_{\phi\in\hpr} m_\phi^2 + \dd(\gamma)+\w(\eta) + \sum_{i\in[n]} p_i \\
	& = \dd(\gamma)+\w(\eta) + \frac{1}{2}\sum_{i\in[n]}\sum_{j\in[n^\prime]}\gamma(i,j)^2 + \sum_{i\in[n]}\codim_\C(\Omega_{\bfs{m}^i}(Q(i,\bullet));\M_\gamma(i,\bullet)) \\
\intertext{which by Proposition \ref{prop:codim.eta} is equal to}
	&= \dd(\gamma)+\w(\eta) + \frac{1}{2}\sum_{i\in[n]}\sum_{j\in[n^\prime]}\gamma(i,j)^2 + \codim_\C(\eta;\M_\gamma).
\end{align*}
Since Kostant series correspond to horizontal strata, the expression in \eqref{eqn:orig.LHS} is really a sum over $\eta \in \Hor(\M_\gamma)$ and can now be rewritten as
\begin{multline}
(-1)^s \, q^{\dd(\gamma) + \frac{1}{2}\sum_{i\in[n]}\sum_{j\in[n^\prime]} \gamma(i,j)^2} \, \times \\
\left(\yy_{\hhs(1)}^\gamma\,\yy_{\hts(1)}^\gamma\right) \cdots \left(\yy_{\hhs(n)}^\gamma \, \yy_{\hts(n)}^\gamma\right)  \sum_{\eta\in\Hor(\M_\gamma)}q^{\w(\eta)+\codim_\C(\eta;\M_\gamma)}\,\cP_\eta.
\nonumber
\end{multline}
After applying Proposition \ref{prop:full.hhs.times.hts}, we obtain that this is further equal to
\begin{align}
 &(-1)^s \, q^{\dd(\gamma) + \mathbf{up}(\gamma) + \frac{1}{2}\sum_{i\in[n]}\sum_{j\in[n^\prime]} \gamma(i,j)^2} \,
\yy_{\hhs}^\gamma\,\yy_{\hts}^\gamma \sum_{\eta\in\Hor(\M_\gamma)}q^{\w(\eta)+\codim_\C(\eta;\M_\gamma)}\,\cP_\eta
\nonumber
\\
\label{eqn:final.LHS}
= & (-1)^s \, q^{ - \hip(\gamma) + \frac{1}{2}\sum_{i\in[n]}\sum_{j\in[n^\prime]} \gamma(i,j)^2} \,
\yy_{\hhs}^\gamma\,\yy_{\hts}^\gamma \sum_{\eta\in\Hor(\M_\gamma)}q^{\w(\eta)+\codim_\C(\eta;\M_\gamma)}\,\cP_\eta.
\end{align}
The analogous computations on the right-hand side of Equation \eqref{eqn:mt} produce
\[
(-1)^s \, q^{ - \vip(\gamma) + \frac{1}{2}\sum_{i\in[n]}\sum_{j\in[n^\prime]} \gamma(i,j)^2} \,
\yy_{\vhs}^\gamma\,\yy_{\vts}^\gamma \sum_{\theta\in\Ver(\M_\gamma)}q^{\w(\theta)+\codim_\C(\theta;\M_\gamma)}\,\cP_\theta.
\]
which by Lemma \ref{lem:switch.HH.HT} becomes
\begin{equation}
\label{eqn:final.RHS}
(-1)^s \, q^{ - \hip(\gamma) + \frac{1}{2}\sum_{i\in[n]}\sum_{j\in[n^\prime]} \gamma(i,j)^2} \,
\yy_{\hhs}^\gamma\,\yy_{\hts}^\gamma \sum_{\theta\in\Ver(\M_\gamma)}q^{\w(\theta)+\codim_\C(\theta;\M_\gamma)}\,\cP_\theta.
\end{equation}
Finally, Corollary \ref{cor:poinc.h.equals.poinc.v} guarantees that \eqref{eqn:final.LHS} and \eqref{eqn:final.RHS} are equal. Since this must hold for every dimension vector $\gamma$, we have established the equality in \eqref{eqn:mt}.
\end{proof}

\bibliographystyle{amsalpha}
\bibliography{jmabib}

\def\cprime{$'$}
\providecommand{\bysame}{\leavevmode\hbox to3em{\hrulefill}\thinspace}
\providecommand{\MR}{\relax\ifhmode\unskip\space\fi MR }
\providecommand{\MRhref}[2]{%
  \href{http://www.ams.org/mathscinet-getitem?mr=#1}{#2}
}
\providecommand{\href}[2]{#2}
\begin{thebibliography}{RWY18}

\bibitem[AB83]{marb1983}
M.~F. Atiyah and R.~Bott, \emph{The {Y}ang-{M}ills equations over {R}iemann
  surfaces}, Philos. Trans. Roy. Soc. London Ser. A \textbf{308} (1983),
  no.~1505, 523--615. \MR{702806}

\bibitem[AGP02]{AGP}
M.~Aguilar, S.~Gitler, and C.~Prieto, \emph{Algebraic topology from a
  homotopical viewpoint}, Universitext, Springer-Verlag, New York, 2002,
  Translated from the Spanish by Stephen Bruce Sontz. \MR{1908260}

\bibitem[DWZ08]{hdjwaz2008}
H.~Derksen, J.~Weyman, and A.~Zelevinsky, \emph{Quivers with potentials and
  their representations. {I}. {M}utations}, Selecta Math. (N.S.) \textbf{14}
  (2008), no.~1, 59--119. \MR{2480710}

\bibitem[DWZ10]{hdjwaz2010}
\bysame, \emph{Quivers with potentials and their representations {II}:
  applications to cluster algebras}, J. Amer. Math. Soc. \textbf{23} (2010),
  no.~3, 749--790. \MR{2629987}

\bibitem[FK94]{lfrk1994}
L.~D. Faddeev and R.~M. Kashaev, \emph{Quantum dilogarithm}, Modern Phys. Lett.
  A \textbf{9} (1994), no.~5, 427--434. \MR{1264393}

\bibitem[FR02]{lfrr2002.duke}
L.~Feh{\'e}r and R.~Rim{\'a}nyi, \emph{Classes of degeneracy loci for quivers:
  the {T}hom polynomial point of view}, Duke Math. J. \textbf{114} (2002),
  no.~2, 193--213. \MR{1920187 (2003j:14005)}

\bibitem[Gab72]{pg1972}
P.~Gabriel, \emph{Unzerlegbare {D}arstellungen. {I}}, Manuscripta Math.
  \textbf{6} (1972), 71--103; correction, ibid. 6 (1972), 309. \MR{0332887 (48
  \#11212)}

\bibitem[Hat02]{ah2002}
A.~Hatcher, \emph{Algebraic topology}, Cambridge University Press, Cambridge,
  2002. \MR{1867354}

\bibitem[Kaz97]{mk1997}
M.~{\'E}. Kazarian, \emph{Characteristic classes of singularity theory}, The
  {A}rnold-{G}elfand mathematical seminars, Birkh\"auser Boston, Boston, MA,
  1997, pp.~325--340. \MR{1429898 (97m:57037)}

\bibitem[Kel11]{bk2011}
B.~Keller, \emph{On cluster theory and quantum dilogarithm identities},
  Representations of algebras and related topics, EMS Ser. Congr. Rep., Eur.
  Math. Soc., Z\"urich, 2011, pp.~85--116. \MR{2931896}

\bibitem[Kel13a]{bk2013}
\bysame, \emph{The periodicity conjecture for pairs of {D}ynkin diagrams}, Ann.
  of Math. (2) \textbf{177} (2013), no.~1, 111--170. \MR{2999039}

\bibitem[Kel13b]{bk2013.fpsac}
\bysame, \emph{Quiver mutations and combinatorial {DT}-invariants}, corrected
  version of a contribution to DMTCS Proceedings: FPSAC 2013, 2013.

\bibitem[Kos59]{bk1959}
B.~Kostant, \emph{A formula for the multiplicity of a weight}, Trans. Amer.
  Math. Soc. \textbf{93} (1959), 53--73. \MR{0109192 (22 \#80)}

\bibitem[KS11]{mkys2011}
M.~Kontsevich and Y.~Soibelman, \emph{Cohomological {H}all algebra, exponential
  {H}odge structures and motivic {D}onaldson-{T}homas invariants}, Commun.
  Number Theory Phys. \textbf{5} (2011), no.~2, 231--352. \MR{2851153
  (2012k:14079)}

\bibitem[KS14]{mkys2014}
\bysame, \emph{Wall-crossing structures in {D}onaldson-{T}homas invariants,
  integrable systems and mirror symmetry}, Homological mirror symmetry and
  tropical geometry, Lect. Notes Unione Mat. Ital., vol.~15, Springer, Cham,
  2014, pp.~197--308. \MR{3330788}

\bibitem[MS74]{jmjs1974}
J.~W. Milnor and J.~D. Stasheff, \emph{Characteristic classes}, Princeton
  University Press, Princeton, N. J.; University of Tokyo Press, Tokyo, 1974,
  Annals of Mathematics Studies, No. 76. \MR{0440554}

\bibitem[Rei10]{mr2010}
M.~Reineke, \emph{Poisson automorphisms and quiver moduli}, J. Inst. Math.
  Jussieu \textbf{9} (2010), no.~3, 653--667. \MR{2650811}

\bibitem[Rim13]{rr2013}
R.~Rim\'{a}nyi, \emph{On the cohomological {H}all algebra of {D}ynkin quivers},
  preprint, 2013.

\bibitem[RWY18]{rraway2018}
R.~Rim\'anyi, A.~Weigandt, and A.~Yong, \emph{Partition identities and quiver
  representations}, J. Algebraic Combin. \textbf{47} (2018), no.~1, 129--169.
  \MR{3757151}

\bibitem[Sch14]{rs2014}
R.~Schiffler, \emph{Quiver representations}, CMS Books in Mathematics/Ouvrages
  de Math\'ematiques de la SMC, Springer, Cham, 2014. \MR{3308668}

\bibitem[Zag88]{dz1988}
D.~Zagier, \emph{The remarkable dilogarithm}, J. Math. Phys. Sci. \textbf{22}
  (1988), no.~1, 131--145. \MR{940391}

\bibitem[Zag07]{dz2007}
\bysame, \emph{The dilogarithm function}, Frontiers in number theory, physics,
  and geometry. {II}, Springer, Berlin, 2007, pp.~3--65. \MR{2290758}

\end{thebibliography}

\end{document}